\def\submitversion{2}%
\if\submitversion2
\documentclass[onecolumn,draftclsnofoot,12pt]{IEEEtran}
\usepackage[square,sort,comma,numbers]{natbib}
\fi
\if\submitversion1
\documentclass[twoside]{article}
\usepackage[accepted]{aistats2022}

\usepackage[round]{natbib}

\fi

\usepackage{enumitem}
\usepackage{amssymb}
\newlist{todolist}{itemize}{2}
\setlist[todolist]{label=$\square$}
\usepackage{amsmath}
\usepackage{algorithm}
\usepackage{algorithmic}
\usepackage{amsmath}
\usepackage{amsthm}
\usepackage{authblk}
\usepackage{microtype}
\usepackage{graphicx}
\usepackage{subfigure}
\usepackage{booktabs} %
\usepackage{bm}
\usepackage{graphics}
\usepackage{tikz}
\usetikzlibrary{arrows}
\usetikzlibrary{positioning,arrows,shapes,chains,fit,scopes}
\usepackage{pgfplots}
\usepackage{pgfplotstable}
\usepackage{multido}
\usepackage{pstricks}
\usepackage{mwe} %
\usepackage{hyperref}
\usepackage{mathtools}
\usepackage{algorithm}
\usepackage{algorithmic}
\usepackage{listings}
\usepackage{lipsum}

\RequirePackage{listings}
\RequirePackage{xcolor}
\definecolor{dkgreen}{rgb}{0,0.6,0}
\definecolor{gray}{rgb}{0.5,0.5,0.5}
\definecolor{mauve}{rgb}{0.58,0,0.82}
\DeclareMathOperator{\diag}{diag}
\newtheorem{theorem}{Theorem}
\newtheorem{proposition}{Proposition}
\newtheorem{definition}{Definition}
\newtheorem{remark}{Remark}
\newtheorem{lemma}{Lemma}
\newtheorem{assumption}{Assumption}
\newtheorem*{theorem*}{Theorem}

\newcommand{\trans}{^{\mathrm T}}

\newcommand{\diff}{\,\mathrm{d}}

\DeclarePairedDelimiterX{\inp}[2]{\langle}{\rangle}{#1, #2}

\definecolor{longhorn}{rgb}{0.8, 0.33, 0.0}

\newcommand\powerP[1]{\left( #1 \right)^{1/p}}
\newcommand{\hmu}{\hat{\mu}}
\newcommand{\hnu}{\hat{\nu}}
\newcommand{\KPW}{\mathcal{KP}W}
\newcommand{\Tr}{\mathrm{Tr}}
\newcommand{\Te}{\mathrm{Te}}
\newcommand{\mnstd}[2]{#1{$\pm$#2}}
\hypersetup{hidelinks}

\if\submitversion2
\title{Two-Sample Test with Kernel Projected Wasserstein Distance}
\author{
Jie Wang, Rui Gao, Yao Xie
\thanks{J.~Wang and Y.~Xie are with H. Milton Stewart School of Industrial and Systems Engineering, Georgia Institute of Technology.
R.~Gao is with Department of Information, Risk, and Operations Management, University of Texas at Austin.}
}
\fi 

\pgfplotsset{compat=1.14}
\begin{document}
\if\submitversion2
\maketitle
\fi 

\if\submitversion1
\twocolumn[

\aistatstitle{Two-Sample Test with Kernel Projected Wasserstein Distance}

\aistatsauthor{ Jie Wang \And Rui Gao \And  Yao Xie }

\aistatsaddress{ Georgia Institute of Technology \And  University of Texas at Austin \And Georgia Institute of Technology } ]
\fi

\begin{abstract}%
We develop a kernel projected Wasserstein distance for the two-sample test, an essential building block in statistics and machine learning: given two sets of samples, to determine whether they are from the same distribution. 
This method operates by finding the nonlinear mapping in the data space which maximizes the distance between projected distributions.
In contrast to existing works about projected Wasserstein distance, the proposed method circumvents the curse of dimensionality more efficiently.
We present practical algorithms for computing this distance function together with the non-asymptotic uncertainty quantification of empirical estimates.
Numerical examples validate our theoretical results and demonstrate good performance of the proposed method.
\end{abstract}

\section{
INTRODUCTION
}
As a fundamental problem in statistical inference \citep{young2005essentials}, two-sample hypothesis testing aims to determine whether two sets of samples come from the same distribution or not.
This problem has broad applications in scientific discovery fields.
For example, it can be applied in anomaly detection~\citep{Chandola2009, savage2014anomaly, ahmed2016survey} to identify abnormal observations that follow a distinct distribution compared with typical observations.
Similarly, in change-point detection~\citep{Vincent08,xie2020sequential,Liyanchange_21}, two-sample testing is essential to detect abrupt changes in streaming data. 
Other notable examples include model criticism~\citep{Lloyd15, chwialkowski2016kernel, bikowski2018demystifying}, causal inference~\citep{lopezpaz2018revisiting}, and health care~\citep{Schober19}.

Parametric or low-dimensional testing scenarios have been the main focus in classical literature.
When extra knowledge about the data distributions is available, one can design parametric tests, such as Hotelling's two-sample test~\citep{hotelling1931}, Student's t-test~\citep{PFANZAGL96}, etc.
Non-parametric two-sample tests are more attractive when the exact parametric form of the data distributions is hard to specify.
It is popular to design non-parametric tests using integral probability metrics, since the evaluation of the corresponding test statistics can be obtained based on samples without knowing the densities of data distributions.
Some earlier works design tests using Kolmogorov-Smirnov distance~\citep{Pratt1981, Frank51}, total variation distance~\citep{Ga1991}, and Wasserstein distance~\citep{delbarrio1999, ramdas2015wasserstein}.
However, it is not proper to use these tests for high-dimensional settings since the sample complexity for estimating those distance functions based on empirical samples suffers from the curse of dimensionality.

There is a strong need for developing non-parametric tests for high-dimensional data, especially for modern applications. A notable contribution is the two-sample test based on Maximum Mean Discrepancy~(MMD)~\citep{Gretton09, Gretton12, cheng2021kernelmanifold}.
Although the power of MMD test with the median choice of kernel bandwidth decays quickly when the dimension of distributions increases~\citep{reddi2014decreasing}, this test with properly chosen bandwidth does not have the curse of dimensionality issue for low-dimensional manifold data as pointed out in \cite{cheng2021kernelmanifold}.
Unfortunately, the MMD test with optimized bandwidth still does not demonstrate good testing power for the small-sampled case as demonstrated numerically in this paper.
In addition, recent works \citep{wang2020twosample, xie2020sequential} leverage the idea of dimensionality reduction for dealing with high-dimensional settings, which use the projected Wasserstein distance as the test statistic, i.e., the test statistic works by finding the linear projector such that the distance between projected distributions is maximized. However, a linear projector may not serve as an optimal design for maximizing the power of tests as demonstrated numerically in Section~\ref{Sec:experiment}.

In this paper, we present a new non-parametric two-sample test statistic aiming for the high-dimensional setting based on a \emph{kernel projected Wasserstein~(KPW) distance}, with a nonlinear projector based on the reproducing kernel Hilbert space~(RKHS) designed to optimize the test power via maximizing the probability distance between the distributions after projection. In addition, our contributions include the following:
\begin{itemize}
\item
We develop a computationally efficient algorithm for evaluating the KPW using a representer theorem to reformulate the problem into a finite-dimensional optimization problem and a block coordinate descent optimization algorithm
which is guaranteed to find an $\epsilon$-stationary point with complexity $\mathcal{O}\left(\epsilon^{-3}\right)$.
\item
To quantify the false detection rate, which is essential in setting the detection threshold, we develop non-asymptotic bounds for empirical KPW distance based on the covering number argument.%
\item 
We present numerical experiments to validate our theoretical results as well as demonstrate the competitive performance of our proposed test using both synthetic and real data.
\end{itemize}

\noindent\textbf{Related Work.}
It is helpful to understand the structure of high-dimension distributions by low-dimensional projections.
Notable methodologies include the principal component analysis~(PCA) \citep{Jolliffe1986}, kernel PCA~\citep{scholkopf1998nonlinear}, factor analysis~\citep{cudeck2000exploratory}, etc.
Several works leverage this idea to design tests for high-dimensional data.
\cite{mueller2015principal} and \cite{xie2020sequential} first design tests by finding the worst-case linear projector that maximizes the distance between projected sample points in one dimension.
Later \cite{lin2020projection2} and \cite{wang2020twosample} naturally extend this idea by developing a projector that maps sample points into $d$ dimensional linear subspace with $d\ge1$, called projected Wasserstein distance.
Efficient optimization algorithms and statistical properties of this distance have been investigated in recent works \citep{huang2021riemannian, lin2020projection}.
However, a linear projector cannot efficiently capture features from data with nonlinear patterns, limiting the performance of tests mentioned above for practical applications.
It is therefore promising to use nonlinear dimensionality reduction for two-sample testing.
Although nonlinear projectors can be obtained using neural networks \citep{genevay2018learning}, the sample complexity of the corresponding test statistic will have slow convergence rates since the neural network function class usually has high complexity in terms of the covering number.
Recently kernel method has been demonstrated to be beneficial for understanding data~\citep{Minh11,brouard2011semi,HQuang13,kadri2013functional} because of sharp sample complexity rate, low computational cost, and flexible representation of features.
This fact motivates us to use a nonlinear projector based on kernels to design tests.
Compared with the linear projector, computing the corresponding statistic and analyzing its performance is more challenging since the function space cannot be parameterized by finite-dimensional coefficients.
We leverage the kernel trick to finish these two parts.

The remaining of this paper is organized as follows.
Section~\ref{Sec:setup} introduces some preliminary knowledge on two-sample testing and related probability distances,
Section~\ref{Sec:algorithm} outlines a practical algorithm for computing KPW distance,
Section~\ref{Sec:inference} studies the uncertainty quantification of empirical KPW distance,
Section~\ref{Sec:experiment} demonstrates some numerical experiments,
and Section~\ref{Sec:discussion} presents some concluding remarks.

\section{
PROBLEM SETUP
}\label{Sec:setup}
Let $x^n:=\{x_i\}_{i=1}^n$ and $y^m:=\{y_i\}_{i=1}^m$ be i.i.d. samples generated from distributions $\mu$ and $\nu$ supported on $\mathbb{R}^D$, respectively.
Our goal is to design a two-sample test which, given samples $x^n$ and $y^m$, decides to accept the null hypothesis $H_0:~\mu=\nu$ or reject $H_0$ in favor of the alternative hypothesis $H_1:~\mu\ne\nu$.
Denote by $T:~(x^n, y^m)\to\{t_0, t_1\}$ the two-sample test, where $t_0$ means we reject $H_1$ and $t_1$ means we accept $H_1$ and reject $H_0$.
Define the type-I risk as the probability of rejecting hypothesis $H_0$ when it is true, and the type-II risk as the probability of accepting $H_0$ when $\mu\ne\nu$:
\begin{align*}
\epsilon^{(\text{I})}_{n,m}&={\mathbb{P}}_{x^n\sim\mu, y^m\sim\nu}
\bigg(
T(x^n,y^m)=t_1
\bigg),\quad\text{under }H_0,\\
\epsilon^{(\text{II})}_{n,m}&={\mathbb{P}}_{x^n\sim\mu, y^m\sim\nu}
\bigg(
T(x^n,y^m)=t_0
\bigg),\quad\text{under }H_1.
\end{align*}
Given parameters $\alpha,\beta\in(0,\frac{1}{2})$, we aim at building a two-sample test such that, when applied to $n$-observation samples $x^n$ and $m$-observation samples $y^m$, it has the type-I risk at most $\alpha$ (i.e., at level $\alpha$) and the type-II risk at most $\beta$ (i.e., of power $1-\beta$).
Moreover, we want to ensure these specifications with sample sizes $n,m$ as small as possible.

We propose a non-parametric test by considering the probability distance functions between two empirical distributions constructed from observed samples. 
Specifically, we design a test $T$ such that the null hypothesis $H_0$ is rejected when 
\[
\mathcal{D}(\hat{\mu}_n,\hat{\nu}_m)>\chi,
\]
where $\mathcal{D}(\cdot,\cdot)$ is a divergence quantifying the differences of two distributions, $\chi$ is a data-dependent threshold, and $\hat{\mu}_n$ and $\hat{\nu}_m$ are empirical distributions from $n$ samples in $\mu$ and $m$ samples in $\nu$, respectively.
Several existing tests can be unified into this framework by taking $\mathcal{D}(\cdot,\cdot)$ as some special probability distances, including the MMD test, total variation distance test, etc.
In this paper, we will design the divergence $\mathcal{D}$ based on the Wasserstein distance, and we specify the cost function $c(x,y)=\|x-y\|_2^2$.

\begin{definition}[Wasserstein Distance]
Given two distributions ${\mu}$ and ${\nu}$, the Wasserstein distance is defined as
\[
W(\mu,\nu)
=
\min_{\pi\in\Pi(\mu,\nu)}~\int c(x,y)
\diff \pi(x,y),
\]
where $c(\cdot,\cdot)$ denotes the cost function quantifying the distance between two points, and $\Pi(\mu,\nu)$ denotes the joint distribution with marginal distributions $\mu$ and $\nu$.
\end{definition}

Although Wasserstein distance has wide applications in machine learning, the finite-sample convergence rate of Wasserstein distance between empirical distributions is slow in high-dimensional settings~\citep{fournier2013rate}.
Therefore, it is not suitable for high-dimensional two-sample tests.
Instead, existing works use the projection idea to rescue this issue.
\begin{definition}[Projected Wasserstein Distance]
Given two distributions ${\mu}$ and ${\nu}$, define the projected Wasserstein distance as
\[
\mathcal{P}W(\mu, \nu)
=
\max_{\substack{\mathcal{A}:~\mathbb{R}^D\to\mathbb{R}^d, 
A\trans A=I_d
}}
W\left(
\mathcal{A}\#\mu,
\mathcal{A}\#\nu
\right),
\]
where the operator $\#$ denotes the push-forward operator, i.e.,
\[
\mathcal{A}(z)\sim \mathcal{A}\#\mu\quad\text{for }z\sim\mu,
\]
and we denote $\mathcal{A}$ as a linear operator such that $\mathcal{A}(z)=A\trans z$ with $z\in\mathbb{R}^D$ and $A\in\mathbb{R}^{D\times d}$.
\end{definition}
This idea is demonstrated to be useful for breaking the curse of dimensionality for the original Wasserstein distance~\citep{lin2020projection, wang2020twosample}.
However, a linear projector is not an optimal choice for dimensionality reduction.
Instead, we will consider a nonlinear projector to obtain a more powerful two-sample test, and we use functions in vector-valued reproducing kernel Hilbert space~(RKHS) for projection.

\begin{definition}[Vector-valued RKHS]
A function $K:~\mathbb{R}^D\times\mathbb{R}^D\to\mathbb{R}^{d\times d}$ is said to be a positive semi-definite kernel if
\[
\sum_{i=1}^N\sum_{j=1}^N\inp{\bar{y}_i}{K(\bar{x}_i,\bar{x}_j)\bar{y}_j}\ge0
\]
for any finite set of points $\{\bar{x}_i\}_{i=1}^N$ in $\mathbb{R}^D$ and $\{\bar{y}_i\}_{i=1}^N$ in $\mathbb{R}^d$.
Given such a kernel, there exists an unique $\mathbb{R}^d$-valued Hilbert space $\mathcal{H}_K$ with the reproducing kernel $K$.
For fixed $x\in\mathbb{R}^D$ and $y\in\mathbb{R}^d$, define the kernel section $K_x$ with the action $y$ as the mapping $K_xy:~\mathbb{R}^D\to\mathbb{R}^d$ such that
\[
(K_xy)(x') = K(x',x)y,\ \quad \forall x'\in\mathbb{R}^D.
\]
In particular, the Hilbert space $\mathcal{H}_K$ satisfies the reproducing property:
\[
\forall f\in\mathcal{H}_K,\quad
\inp{f}{K_xy}_{\mathcal{H}_K}=\inp{f(x)}{y}.
\]
\end{definition}

\begin{definition}[Kernel Projected Wasserstein Distance]
Consider a $\mathbb{R}^d$-valued RKHS $\mathcal{H}$ with the corresponding kernel function $K$.
Given two distributions ${\mu}$ and ${\nu}$, define the kernel projected Wasserstein~(KPW) distance as
\[
\begin{aligned}
\mathcal{KP}W(\mu, \nu)
&=
\max_{f\in\mathcal{F}}~W\left(
f\#\mu,
f\#\nu
\right)
\end{aligned}
\]
where the function class $\mathcal{F}=\{f\in\mathcal{H}:~\|f\|_{\mathcal{H}}\le 1\}$.
\end{definition}

\begin{remark}\label{Remark:KPW:choose:kernel}
For $d=1$, when the kernel function $K(x,y)=\inp{x}{y}$, the KPW distance reduces into the PW distance.
However, these two distances are not the same for general $d$.
Moreover, existing works \citep{Minh11, Micchelli05, Andrea08, Baldassarre10} consider the design of the matrix-valued kernel function for $d>1$ as
\begin{equation}\label{Eq:M:v:RKHS}
K(x,x')=k(x,x')\cdot P,
\end{equation}
where $k(\cdot,\cdot)$ denotes a scalar-valued kernel function and $P\in\mathbb{R}^{d\times d}$ is a positive semi-definite matrix that encodes the relation between the output space.
Such a design reduces the computational cost for applying vector-valued RKHS.
\end{remark}

In this paper, we design the two-sample test as follows.
We split the data points into training and testing datasets.
We first use the training set to train a nonlinear projector that maps data points into $\mathbb{R}^d$-subspace,
and then perform the permutation test on testing data points that are projected based on the trained projector.
The detailed algorithm is presented in Algorithm~\ref{Alg:permutation:test}.
This test is guaranteed to exactly control the type-I error~\citep{good2013permutation} because we evaluate the $p$-value of the test via the permutation approach.
To obtain reliable two-sample tests, we also require the KPW distance satisfies the discriminative property that $\KPW(\mu,\nu)=0$ if and only if $\mu=\nu$.
The following proposition reveals that this property holds by considering the vector-valued RKHS satisfying the universal property, the proof of which is provided in Appendix~\ref{app:theorem:sec:setup}.
We also study how to compute the kernel projected distance and its related statistical properties in the following sections.

\begin{center}
\begin{algorithm}[!t]
\caption{
Permutation two-sample test using the KPW distance
} 
\begin{algorithmic}[1]\label{Alg:permutation:test}
\REQUIRE
{Level $\alpha$,
number of permutation times $N_p$,
collected samples $x^n$ and $y^m$.
}
\STATE{ 
Split data as $x^n = x^{\Tr}\cup x^{\Te}$ and $y^m = y^{\Tr} \cup y^{\Te}$.
}
\STATE{Formulate empirical distributions $(\hmu^{\Tr}, \hnu^{\Tr})$ corresponding to $(x^{\Tr}, y^{\Tr})$.}
\STATE{
Obtain $f$ as the (approximate) optimal projector to $\KPW(\hmu^{\Tr}, \hnu^{\Tr})$.
}
\STATE{Compute the statistic $T = W(f\#\hmu^{\Te}, f\#\hnu^{\Te})$.}
\FOR{$t = 1, \ldots, N_p$} 
\STATE{Shuffle $x^{\Te}\cup y^{\Te}$ to obtain $x^{\Te}_{(t)}$ and $y^{\Te}_{(t)}$.}
\STATE{Formulate empirical distributions $(\hmu^{\Te}_{(t)}, \hnu^{\Te}_{(t)})$ corresponding to $(x^{\Te}, y^{\Te})$.}
\STATE{Compute the statistic for permuted samples $T_t=W(f\#\hmu^{\Te}_{(t)}, f\#\hnu^{\Te}_{(t)})$.}
\ENDFOR
\\
\textbf{Return} the $p$-value $\frac{1}{N_p}\sum_{t=1}^{N_p}1\{T_t\ge T\}$.
\end{algorithmic}
\end{algorithm}
\end{center}

\begin{proposition}[Discriminative Property of KPW]\label{Proposition:discriminant}
Denote by $\mathcal{C}_b(\mathcal{X})$ the space of bounded and continuous $\mathbb{R}^d$-valued functions on $\mathcal{X}$.
Assume that $\mathcal{H}$ is a universal vector-valued RKHS so that for any $\varepsilon>0$ and $f\in\mathcal{C}_b(\mathcal{X})$, there exists $g\in\mathcal{H}$ so that 
\[
\|f-g\|_{\infty} \triangleq \sup_{x\in\mathcal{X}}\|f(x)-g(x)\|_{2}<\varepsilon.
\]
Then the KPW distance $\mathcal{KP}W(\mu,\nu)=0$ if and only if $\mu=\nu$.
\end{proposition}

\section{
COMPUTING KPW DISTANCE
}
\label{Sec:algorithm}
By the definition of Wasserstein distance, computing $\mathcal{KP}W(\hat{\mu}_n, \hat{\nu}_m)$ is equivalent to the following max-min problem:
\begin{equation}\label{Eq:nominal:KPW}
\max_{f\in\mathcal{H}:~\|f\|_{\mathcal{H}}^2\le 1}~
\left\{
\min_{\pi\in\Gamma}~\sum_{i,j}\pi_{i,j}\|f(x_i) - f(y_j)\|^2_2
\right\},
\end{equation}
where $\Gamma=\left\{\pi\in\mathbb{R}^{n\times m}_+:~ \sum_j\pi_{i,j}=\frac{1}{n}, \sum_i\pi_{i,j}=\frac{1}{m}\right\}$.

The computation of KPW distance has numerous challenges.
It is crucial to design a suitable kernel function to obtain low computational complexity and reliable testing power, which will be discussed in Section~\ref{Sec:experiment}.
Moreover, the function $f\in\mathcal{H}$ is a countable combination of basis functions, i.e., the problem~\eqref{Eq:nominal:KPW} is an infinite-dimensional optimization.
By developing the representer theorem in Theorem~\ref{Theorem:representer}, we are able to convert this problem into a finite-dimensional problem.
Finally, there is no theoretical guarantee for finding the global optimum since it is a non-convex non-smooth optimization problem. Moreover, Sion's minimax theorem is not applicable because the problem \eqref{Eq:nominal:KPW} is not a convex programming: the inner minimization of quadratic function makes the objective in \eqref{Eq:nominal:KPW} not concave in $f$ in general.
Based on this observation, we only focus on optimization algorithms for finding a local optimum point in polynomial time.

\begin{theorem}[Representer Theorem for KPW Distance]\label{Theorem:representer}
There exists an optimal solution to \eqref{Eq:nominal:KPW} that admits the following expression:
\[
\hat{f} = 
\sum_{i=1}^nK_{x_i}a_{x,i}
-
\sum_{j=1}^mK_{y_j}a_{y,j},
\]
where $K_{x}(\cdot)$ denotes the kernel section and $a_{x,i},a_{y,j}\in\mathbb{R}^d$ for $i=1,\ldots,n, j=1,\ldots,m$ are coefficients to be determined.
\end{theorem}

The proof of Theorem~\ref{Theorem:representer} is provided in Appendix \ref{Appendix:proof:algorithm}, in which standard representer theorem in literature~\citep[Theorem~1]{Bernhard01} is not applicable since the RKHS norm serves as a hard constraint instead of the regularization of the objective function.
In order to express the optimal solution as the compact matrix form, define $a_x\in\mathbb{R}^{nd}$ as the concatenation of coefficients $a_{x,i}$ for $i=1,\ldots,n$ and 
\[
K_{z}(x^n)=
\begin{pmatrix}
K(z,x_1)
&
\cdots
&
K(z,x_n)
\end{pmatrix}\in\mathbb{R}^{d\times nd}.
\]
We also define the vector $a_y$ and matrix $K_{z}(y^m)$ likewise.
Then we have
\[
\hat{f}(z) = K_{z}(x^n)a_x - K_{z}(y^m)a_y, \ \quad \forall z\in\mathcal{X}.
\]
Define the gram matrix $K(x^n,x^n)$ as the $n\times n$ block matrix with the $(i,j)$-th block being $K(x_i,x_j)$. The gram matrices $K(x^n,y^m), K(y^m,x^n)$ and $K(y^m,y^m)$ can be defined likewise.
Denote by $G$ the concatenation of gram matrices:
\[
G=\begin{pmatrix}
K(x^n,x^n)   &   -K(x^n,y^m)\\
-K(y^m,x^n)  &   K(y^m,y^m)
\end{pmatrix},
\]
and we assume that $G$ is positive definite. Otherwise, we add the gram matrix with a small number times identity matrix to make it invertible.
Substituting the expression of $\hat{f}(z), z\in\mathcal{X}$ into \eqref{Eq:nominal:KPW}, we obtain a finite-dimensional optimization problem:
\begin{equation*}
\max_{\omega}~\left\{
\min_{\pi\in\Gamma}~
\sum_{i,j}\pi_{i,j}c_{i,j}:~
\omega\trans G\omega\le 1
\right\},
\end{equation*}
where $\omega=[a_x\trans, a_y\trans]\trans\in\mathbb{R}^{d(n+m)}$, $c_{i,j}= \|A_{i,j}\omega\|_2^2$, and 
\begin{align*}
A_{i,j} &= 
[K_{x_i}(x^n)-K_{y_j}(x^n), K_{y_j}(y^m)-K_{x_i}(y^m)].
\end{align*}

Suppose that the inverse of $G$ admits the Cholesky decomposition $G^{-1}=UU\trans$, then by the change of variable technique $s=U^{-1}\omega$, we obtain the norm-constrained optimization problem:
\begin{equation}
\label{Eq:finite}
\max_{s\in\mathbb{R}^{d(n+m)}}~\left\{
\min_{\pi\in\Gamma}~
\sum_{i,j}\pi_{i,j}c_{i,j}:~
s\trans s\le 1
\right\},
\end{equation}
and we can replace the constraint $s\trans s\le 1$ with $s\trans s=1$ based on the fact that the norm function satisfies the linear property.
In other words, the decision variable $s$ belongs to the Euclidean ball $\mathbb{S}^{d(n+m)-1}=\{s\in\mathbb{R}^{d(n+m)}: s\trans s=1\}$.

For the ease of optimization, we consider the entropic regularization of the problem~\eqref{Eq:finite}:
\begin{equation}
\label{Eq:max:min:entropy}
\max_{s\in\mathbb{S}^{d(n+m)-1}}~\left\{\min_{\pi\in\Gamma}
~
\sum_{i,j}\pi_{i,j}
c_{i,j}
-\eta H(\pi)\right\},
\end{equation}
in which we denote the entropy function $H(\pi)=-\sum_{i,j}\pi_{i,j}(\log\pi_{i,j}-1)$.
By the duality theory of entropic optimal transport~\citep{genevay2019entropy} and the change-of-variable technique, \eqref{Eq:max:min:entropy} is equivalent to the following minimization problem:
\begin{equation}\label{Eq:final:min:three:block}
\min_{\substack{s\in\mathbb{S}^{d(n+m)-1}, u\in\mathbb{R}^n, v\in\mathbb{R}^m}}~
F(u,v,s),
\end{equation}
where
\begin{align*}
c_{i,j}&=\|A_{i,j}Us\|_2^2,\\
{\pi_{i,j}}(u,v,s)&=\exp\left(
-\frac{1}{\eta}c_{i,j} + u_i + v_j\right),\\
F(u,v,s)&=\sum_{i,j}{\pi_{i,j}}(u,v,s)-\frac{1}{n}\sum_{i=1}^nu_i - \frac{1}{m}\sum_{j=1}^mv_j.
\end{align*}
The details for this deviation is deferred in Appendix~\ref{Appendix:proof:algorithm}.
Based on this formulation, we consider a Riemannian block coordinate descent~(BCD) method~\citep{HildrethBCD} for optimization, which updates a block of variables by minimizing the objective function with respect to that block while fixing values of other blocks:
\begin{subequations}
\begin{align}
u^{t+1}&=\min_{u\in\mathbb{R}^n}F(u, v^t, s^t),
\label{Eq:update:u}
\\
v^{t+1}&=\min_{v\in\mathbb{R}^m}F(u^{t+1}, v, s^t),
\label{Eq:update:v}\\
\zeta^{t+1}&=\sum_{i,j}\nabla_{s}{\pi_{i,j}}(u^{t+1}, v^{t+1}, s^t),
\label{Eq:update:zeta}
\\
\xi^{t+1}&=\mathcal{P}_{s^t}\big(\zeta^{t+1}\big), \label{Eq:update:xi}\\
s^{t+1}&=\text{Retr}_{s^t}\big(-\tau \xi^{t+1}\big),\label{Eq:update:s}
\end{align}
where the operator $\mathcal{P}_s(\zeta)$ denotes the orthogonal projection of the vector $\zeta$ onto the tangent space of the manifold $\mathbb{S}^{d(n+m)-1}$ at $s$:
\[
\mathcal{P}_{s}\big(\zeta\big) = \zeta - \inp{s}{\zeta}s,\ \quad s\in \mathbb{S}^{d(n+m)-1},
\]
and the retraction on this manifold is defined as
\begin{equation}\label{Eq:retraction:step}
    \text{Retr}_{s}\big(-\tau\xi\big) = \frac{s-\tau\xi}{\|s-\tau\xi\|},\ \quad s\in \mathbb{S}^{d(n+m)-1}.
\end{equation}
Note that the update steps \eqref{Eq:update:u} and \eqref{Eq:update:v} have closed-form expressions:
\begin{align}
u^{t+1}&
= u^t + 
\left\{
\log
\frac{1/n}{\sum_j{\pi_{i,j}}(u^t, v^t, s^t)}
\right\}_{i\in[n]},\label{Eq:update:u:closed}\\
v^{t+1}&=
v^t + 
\left\{
\log
\frac{1/m}{\sum_i{\pi_{i,j}}(u^{t+1}, v^t, s^t)}
\right\}_{j\in[m]},\label{Eq:update:v:closed}
\end{align}
and the Euclidean gradient $\zeta^{t+1}$ in \eqref{Eq:update:zeta} can be computed using the chain rule:
\begin{equation}
\zeta^{t+1}=-\frac{1}{\eta}U\trans\left[ \sum_{i,j}\pi_{i,j}(u^{t+1}, v^{t+1}, s^t)A_{i,j}\trans A_{i,j}\right] Us^t.
\label{Eq:update:zeta:closed}
\end{equation}
\end{subequations}

\begin{center}
\begin{algorithm}[!t]
\caption{
BCD Algorithm for Solving \eqref{Eq:final:min:three:block}} 
\begin{algorithmic}[1]\label{Alg:1}
\REQUIRE
{
Empirical distributions $\hat{\mu}_n$ and $\hat{\nu}_m$.
}
\STATE{Initialize $v^0, s^0$}
\FOR{$t = 0,1,2,\ldots,T-1$} 
\STATE{Update $u^{t+1}$ according to~\eqref{Eq:update:u:closed}}
\STATE{Update $v^{t+1}$ according to~\eqref{Eq:update:v:closed}}
\STATE{Update the Euclidean and Riemannian gradient $\zeta^{t+1}$ and $\xi^{t+1}$, according to \eqref{Eq:update:zeta:closed} and \eqref{Eq:update:xi}, respectively.}
\STATE{Update $s^{t+1}$ according to \eqref{Eq:update:s}}
\ENDFOR
\\
\textbf{Return $u^*=u^T, v^*=v^T, s^*=s^T$.}
\end{algorithmic}
\end{algorithm}
\end{center}

The overall algorithm for solving the problem~\eqref{Eq:final:min:three:block} is summarized in Algorithm~\ref{Alg:1}.
We provide details for efficient implementation of the proposed algorithms in Appendix~\ref{Appendix:Implementation}.
We also give a brief introduction to Riemannian optimization in Appendix~\ref{Appendix:manifold}.
The following theorem gives a convergence analysis of our proposed algorithm.%
The proof of this result is provided in Appendix~\ref{Appendix:proof:algorithm}, which follows similar procedure in \cite{huang2021riemannian}.
The main difference lies in establishing the descent lemma for updating the variable $s$ on sphere instead of Stiefel manifold.
Specifically, the procedure for finding the upper bound on the cost function $c_{i,j}$, the Lipschitz constant for $\pi_{i,j}(u,v,s)$ in $s$, and the Lipschitz constants of the retraction operator \eqref{Eq:retraction:step} will be different.

\begin{theorem}[Convergence Analysis for BCD]\label{Theorem:convergence:BCD}
We say that $(\hat{u}, \hat{v}, \hat{s})$ is a $(\epsilon_1,\epsilon_2)$-stationary point of \eqref{Eq:final:min:three:block} if 
\begin{align*}
    \|\text{Grad}_{s}F(\hat{u}, \hat{v}, \hat{s})\|&\le \epsilon_1,\\
    F(\hat{u}, \hat{v}, \hat{s}) - \min_{u,v}F(u,v,\hat{s})&\le \epsilon_2,
\end{align*}
where $\text{Grad}_{s}F(u,v,s)$ denotes the derivative of $F$ with respect to $s$ on the sphere $\mathbb{S}^{d(n+m)-1}$.
Let $\{u^t,v^t,s^t\}$ be the sequence generated by Algorithm~\ref{Alg:1}, then Algorithm~\ref{Alg:1} returns an $(\epsilon_1,\epsilon_2)$-stationary point in 
\[
T = \mathcal{O}\left(
\log(mn)\cdot
\left[ 
\frac{1}{\epsilon_2^3} + \frac{1}{\epsilon_1^2\epsilon_2}
\right]
\right),
\]
iterations, where the notation $O(\cdot)$ hides constants related to the initial guess $(v^0,s^0)$ and the term $\max_{i,j}\|A_{i,j}U\|$.
\end{theorem}

\begin{remark}[Complexity of Algorithm~\ref{Alg:1}]
Denote $N=n\lor m$\footnote{We denote $a\lor b$ for $\max\{a,b\}$ and $a\land b$ for $\min\{a,b\}$.}.
Note that the iteration \eqref{Eq:update:u:closed} and \eqref{Eq:update:v:closed} can be implemented in $O(N)$ iterations.
Second, the retraction step in \eqref{Eq:update:s} requires $O(dN)$ arithmetic operations.
Third, the computation of the Euclidean vector in \eqref{Eq:update:zeta} can be implemented in $O(d^3N^3)$ operations, and the projection step can be done in $O(dN)$ operations.
Therefore, the number of arithmetic operations in each iteration is of $O(d^3N^3)$.
In summary, Algorithm~\ref{Alg:1} returns an $(\epsilon_1,\epsilon_2)$-stationary point in 
\[
\mathcal{O}\left(
d^3N^3
\log(N)\cdot
\left[ 
\frac{1}{\epsilon_2^3} + \frac{1}{\epsilon_1^2\epsilon_2}
\right]
\right)
\]
arithmetic operations.
Note that this computational complexity is independent of the dimension $D$ of samples since we only need to compute the gram matrix as an input.
The storage cost is of $\mathcal{O}(d^2N^2)$, in which the most expensive step is to store the gram matrix $G$.
\end{remark}

\section{
PERFORMANCE GUARANTEES
}\label{Sec:inference}
In this section, we build statistical properties of the empirical KPW distance, though in practice we may not succeed in finding a global optimum solution to the non-convex optimization problem \eqref{Eq:nominal:KPW}.
We assume the cost function for the Wasserstein distance has the form $c(x,y)=\|x-y\|_2^p$ with $p\in[1,\infty)$.
Moreover, results throughout this section are based on the following assumption.
\begin{assumption}\label{Assumption:1}
For any $x,x'\in\mathcal{X}$, the matrix-valued kernel $K(x,x')$ is symmetric and satisfies
\[
0\preceq K(x,x')\preceq BI_d.
\]
\end{assumption}

\begin{definition}[(Projection) Poincare Inequality]
1. A distribution $\mu$ is said to satisfy a Poincare inequality if there exists an $M>0$ for $X\sim \mu$ so that $\text{Var}[f(X)]\le M\mathbb{E}[\|\nabla f(X)\|^2]$ for any $f$ satisfying $\mathbb{E}[f(X)^2]<\infty$ and $\mathbb{E}[\|\nabla f(X)\|^2]<\infty$.\\
2. A distribution $\mu$ is said to satisfy a projection Poincare inequality if there exists an $M>0$ for any $f\in\mathcal{F}$ and $X\sim f\#\mu$ so that $\text{Var}[f(X)]\le M\mathbb{E}[\|\nabla f(X)\|^2]$ for any $f$ satisfying $\mathbb{E}[f(X)^2]<\infty$ and $\mathbb{E}[\|\nabla f(X)\|^2]<\infty$.
\end{definition}

\begin{remark}\label{Remark:poincare}
The Poincare inequality characterizes the relation about the variance of a function and its derivative in the spirit of the Sobolev inequality.
It is a standard technical assumption for investigating the empirical convergence of Wasserstein distance~\citep{lin2020projection,Lei_2020}, and is satisfied for various exponential measures such as the Gaussian distribution.
See \cite{Ledoux19} for more examples.
\end{remark}

\begin{lemma}\label{Lemma:10}
Assume that the distribution $\mu$ satisfies a projection Poincare inequality.
Then
\begin{align*}
    \mathbb{E}&[\powerP{\KPW(\hmu_n, \mu)}]\lesssim
n^{-\frac{1}{(2p)\lor d}}(\log n)^{\zeta_{p,d}/p}\\
&\qquad\qquad\qquad+n^{-1/(2\lor p)}\sqrt{\log(n)} + n^{-1/p}\log(n),
\end{align*}
where $\zeta_{p,d}=1\{d=2p\}$, and $\lesssim$ refers to "less than" with a constant depending only on $(p,B)$.
\end{lemma}
\begin{lemma}\label{Lemma:variance}
Assume that the distribution $\mu$ satisfies a Poincare inequality, and any $f\in\mathcal{F}$ is $L$-Lipschitz.
Then with probability at least $1-\alpha$, it holds that 
\begin{align*}
&\left|\powerP{\KPW(\hmu_n, \mu)} - \mathbb{E}[ \powerP{\KPW(\hmu_n, \mu)}]\right|\\
&\qquad\le\max\left\{
\varrho\log(1/\alpha), \sqrt{\varrho\log(1/\alpha)}
\right\}n^{-1/(2\lor p)}L^{1/p},
\end{align*}
where $\varrho>0$ is a constant that depends on $M$.
\end{lemma}
\begin{figure*}[t]
    \centering
    \includegraphics[width=\textwidth]{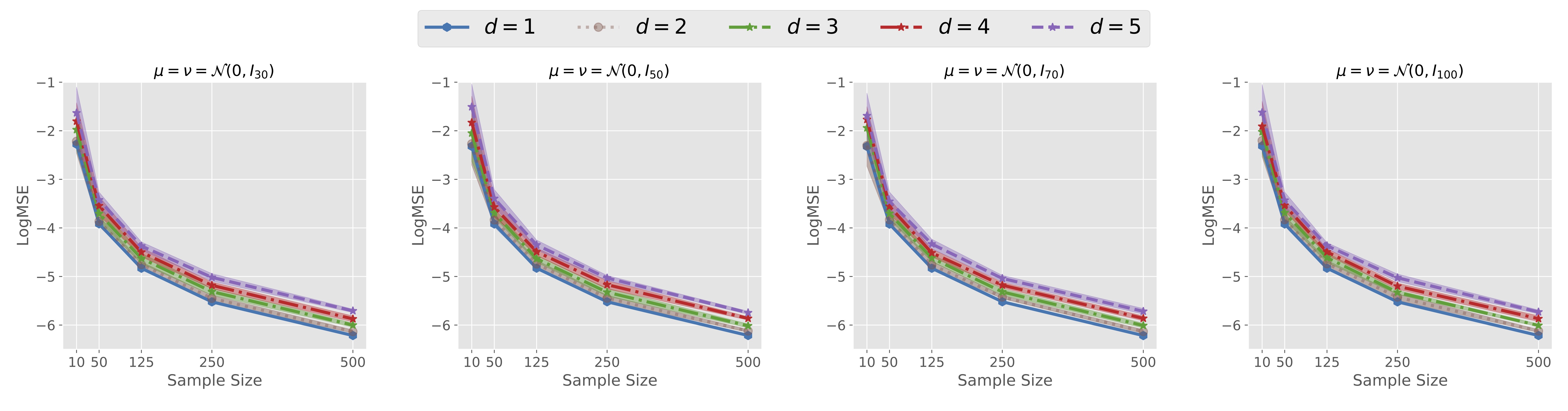}
    \caption{Average values of KPW distances between empirical distributions $\hmu_n$ and $\hnu_n$ as the sample size $n$ varies. Results are averaged for $10$ independent trials and the shaded areas show the corresponding error bars.}
    \label{fig:sample:Gaussian}
\end{figure*}

Proof of two lemmas above follows similar covering number arguments in \cite{lin2020projection}, the details of which are deferred in Appendix~\ref{app:inference:theorem}.
The main difference is that we incorporate the reproducing property of vector-valued RKHS to give a valid bound on the covering number of the RKHS ball $\mathcal{F}$.
Based on these two lemmas and the triangular inequality for Wasserstein distance, we give a finite-sample guarantee for the convergence of the KPW distance in Theorem~\ref{Proposition:finite:KPW:preliminary}.
Compared with the sample complexity of estimating Wasserstein distance, KPW distance does not suffer from the curse of dimensionality as the RKHS ball $\mathcal{F}$ has low complexity.

\begin{theorem}[Finite-sample Guarantee]\label{Proposition:finite:KPW:preliminary}
Suppose the target distributions $\mu=\nu$, which satisfies projection Poincare inequality and Poincare inequality.
Moreover, any $f\in\mathcal{F}$ is $L$-Lipschitz.
Take $N=n\land m$, then with probability at least $1-2\alpha$, it holds that
\begin{align*}
&\powerP{\KPW(\hmu_n, \hnu_m)}\lesssim
N^{-\frac{1}{(2p)\lor d}}(\log N)^{\zeta_{p,d}/p}\\
&\quad+
N^{-1/(2\lor p)}\sqrt{\log(N)} + N^{-1/p}\log(N)\\
&\quad + \max\left\{
\varrho\log(1/\alpha), \sqrt{\varrho\log(1/\alpha)}
\right\}N^{-1/(2\lor p)}L^{1/p}.
\end{align*}

\end{theorem}

\subsection{Performance Guarantees for \texorpdfstring{$p\in[1,2)$}{1<=p<2}}
When showing concentration results for $p$-Wasserstein distance with $p\in[1,2)$, however, it is not necessary to rely on the Poincare inequality assumption.
The main result for this case is summarized in Theorem~\ref{theorem:p12:UQ} (see details in Appendix~\ref{Sec:sub:finite:sample:1:2}).

\begin{theorem}[Finite-sample Guarantee]\label{theorem:p12:UQ}
Suppose the target distributions $\mu=\nu$.
Then with probability at least $1-2\alpha$, it holds that
\begin{align*}
&\powerP{\KPW(\hmu_n, \nu_m)}\lesssim
N^{-\frac{1}{(2p)\lor d}}(\log N)^{\zeta_{p,d}/p}\\
&\qquad+
N^{1/2-1/p}\sqrt{\log(N)}+ N^{-1/p}\\
&\qquad + N^{1/2-1/p}\sqrt{\log\frac{2}{\alpha}}.
\end{align*}
where $N=n\land m$ and $\lesssim$ refers to "less than" with a constant depending only on $(p,B)$.
\end{theorem}

\subsection{Sample Complexity}
We also numerically examine the sample complexity of the empirical KPW distance $\KPW(\hmu_n,\hnu_n)$ with $\mu=\nu=\mathcal{N}(0,I_D)$, where $n\in\{10, 50, 125, 250, 500\}$ and $D\in\{30,50,70,100\}$.
Figure~\ref{fig:sample:Gaussian} reports the average distances and the shaded areas show the corresponding error bars over $10$ independent trials. 
We defer the detailed experiment setup and the plots of the computation time in Appendix~\ref{Appendix:experiment:sample:detail}.
From the plot we can see that the empirical KPW distances decay to zero quickly when the sample size $n$ increases.
Moreover, the distances with smaller values of $d$ have faster decaying rates.
Finally, the convergence behavior of the empirical KPW distances is nearly independent of the choice of $D$, which alleviates the issue of the curse of dimensionality for the original Wasserstein distance.
These facts confirm the finite-sample guarantee discussed in Theorem~\ref{Proposition:finite:KPW:preliminary}.

\begin{figure*}[t]
    \centering
    \includegraphics[width=0.85\textwidth]{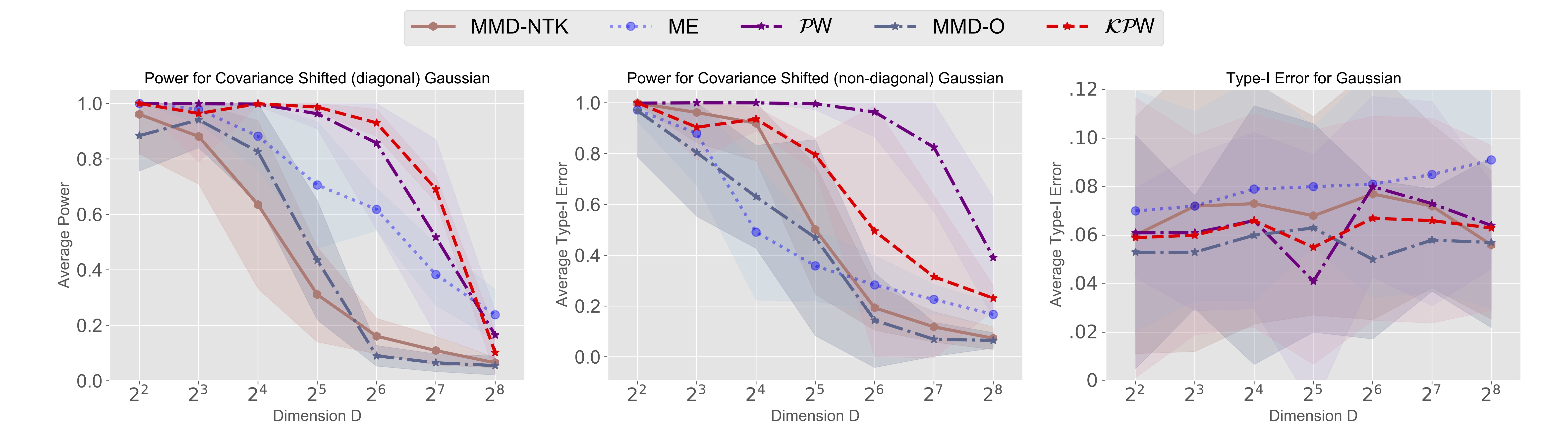}
    \caption{
    Testing results on Gaussian distributions across different choices of dimension $D$.
    Left: power for Gaussian distributions, where the shifted covariance matrix is still diagonal;
    Middle: power for Gaussian distributions, where the shifted covariance matrix is non-diagonal;
    Right: Type-I error.}
    \label{fig:label:Gaussian}
\end{figure*}

\begin{table*}[t]
\centering
  \caption{
Average test power and standard error about detecting distribution abundance change in \emph{MNIST} dataset across different choices of sample size.
   } \label{tab:MNIST_RES1}
\vspace{1mm}
\begin{tabular}{c|>{\centering}p{0.15\textwidth}>{\centering}p{0.15\textwidth}>{\centering}p{0.15\textwidth}>{\centering}p{0.15\textwidth}>{\centering\arraybackslash}p{0.15\textwidth}}
\toprule
$N$ & MMD-NTK & MMD-O & ME & PW & KPW\\
\midrule
\phantom{1}$\,$200 & \mnstd{0.639}{0.029} & \mnstd{\bf 0.696}{0.006} & \mnstd{0.298}{0.031} & \mnstd{0.302}{0.033} & \mnstd{0.663}{0.015} \\
\phantom{2}$\,$250 & \mnstd{0.763}{0.010} & \mnstd{0.781}{0.002} & \mnstd{0.472}{0.017} & \mnstd{0.369}{0.030} & \mnstd{\bf 0.785}{0.014} \\
\phantom{1}$\,$300 & \mnstd{0.813}{0.016} & \mnstd{0.869}{0.002} & \mnstd{0.630}{0.025} & \mnstd{0.524}{0.023} & \mnstd{\bf 0.928}{0.001} \\
\phantom{1}$\,$400 & \mnstd{0.881}{0.013} & \mnstd{0.956}{0.003} & \mnstd{0.779}{0.020} & \mnstd{0.591}{0.044} & \mnstd{\bf 0.978}{0.000} \\
\phantom{1}$\,$500 & \mnstd{0.950}{0.002} & \mnstd{0.988}{0.000} & \mnstd{0.927}{0.006} & \mnstd{0.782}{0.040} & \mnstd{\bf 1.000}{0.000}  \\
\midrule
Avg. & 0.809 & 0.858 & 0.621 & 0.513 & {\bf 0.870} \\
\bottomrule
\end{tabular}
\vspace{-1em}
\end{table*}
\section{
NUMERICAL EXPERIMENTS
}\label{Sec:experiment}
Throughout this section, we compare the performance of tests with the following procedures. (i) PW: the projected Wasserstein test where the projector is a linear mapping~\citep{wang2020twosample}; (ii) MMD-O: the MMD test with a Gaussian kernel whose bandwidth is optimized~\citep{liu2020learning}; (iii) MMD-NTK: the test that combines both neural networks and MMD~\citep{cheng2021neural}; and (iv) ME: the mean embedding test with optimized hyper-parameters~\citep{JitkrittumME}.
Implementation details on those baseline methods are omitted in Appendix~\ref{Appendix:config:baseline}.
When dealing with synthetic datasets, we generate a single sample set as the training set to learn parameters for each method.
Then we evaluate the power of tests on $100$ new sample sets generated from the same distribution.
When dealing with real datasets, we randomly take part of samples as the training set, and evaluate the power on $100$ randomly chosen subsets from the remaining samples.
The number of permutations in Algorithm~\ref{Alg:permutation:test} is set to be $N_p=100$.
We control the type-I error for all tests at $\alpha=0.05$.

When using the KPW distance, we follow \eqref{Eq:M:v:RKHS} to design kernels to decrease the computational complexity. 
More specifically, we choose the scalar-valued kernel $k(\cdot,\cdot)$ to be a standard Gaussian kernel with the bandwidth $\sigma^2$, and 
\[
P=(1-\rho)\bm 1\bm 1\trans + \rho I_d,\quad 
\text{with }\rho\in[0,1].
\]
We use the cross-validation approach to select the hyper-parameters $\rho$ and $\sigma^2$, the details of which are deferred in Appendix~\ref{Appendix:cross:validation}.
The dimension $d$ is pre-specified and fixed into $3$ in all experiments.
We also present a study on the impact of hyper-parameters such as the projected dimension $d$ and regularization parameter $\eta$ in Appendix~\ref{Appendix:impact:hyper:parameter}.

\subsection{Tests for Synthetic Datasets}\label{Sec:Synetic}
We first investigate the performance when $\mu$ and $\nu$ are Gaussian distributions with diagonal covariance matrices.
Specifically, we take $\mu=\mathcal{N}(0,I_D)$ and $\nu=\mathcal{N}(0,\Sigma)$ is the covariance shifted Gaussian, where the matrix $\Sigma=\text{diag}(4,4,4,1,\ldots,1)$.
In other words, we only scale the first three entries of the covariance matrix to make the high-dimensional testing problem challenging to handle.
Fig.~\ref{fig:label:Gaussian} reports the type-I and type-II errors for various tests across different choices of dimension $D$.
We observe that both PW and KPW tests perform the best, while the power for other benchmark methods degrades quickly when the dimension $D$ increases.

Next, we examine the case where $\nu$ has a non-diagonal covariance matrix.
We take $\mu=\mathcal{N}(0,I_D)$ and $\nu=\mathcal{N}(0,V\Sigma V\trans)$, where 
$V$ is an orthogonal matrix with $V_{i,j}=\sqrt{2/(D+1)}\sin(ij\pi/(D+1))$ and $\Sigma=\diag(5,5,5,1,\ldots,1)$.
Testing results for various choices of dimension $D$ is reported in the middle of Fig.~\ref{fig:label:Gaussian}.
In this case, the PW test performs slightly better than the KPW test. 
One possible explanation is that linear mapping seems to be the optimal choice for two-sample testing with covariance shifted Gaussian distributions.
It is promising to design other types of matrix-valued kernel functions to improve performances of the KPW test.

\begin{figure*}[t]
\centering
\includegraphics[width=0.95\textwidth]{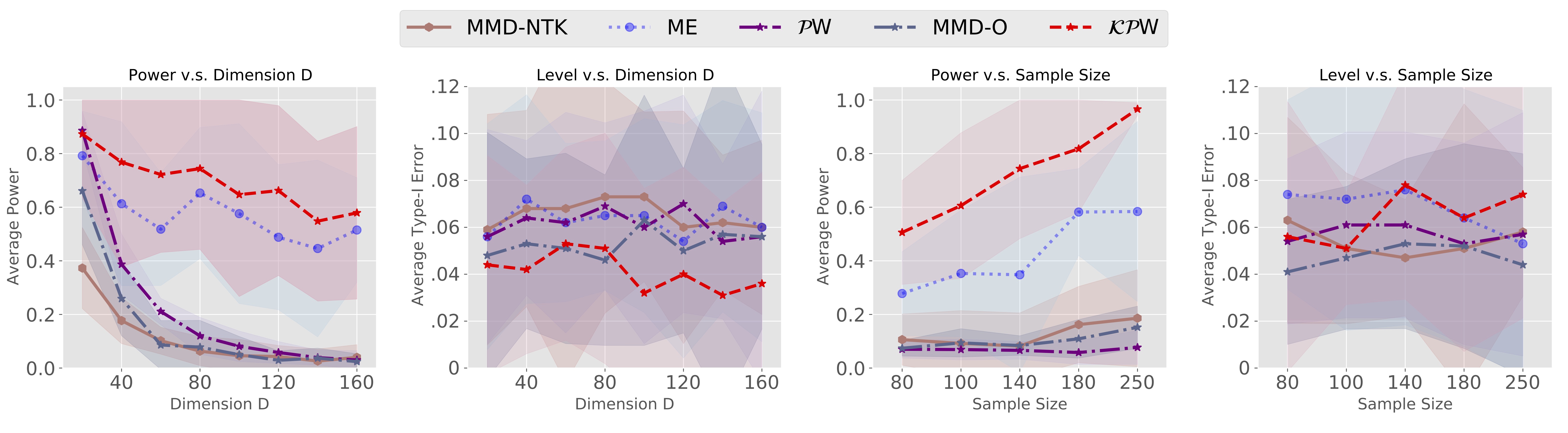}
\caption{
Testing results on Gaussian-mixture distributions.
Left two: type-I and type-II errors across different choices of dimension $D$ with fixed sample size $n=m=200$;
Right two: type-I and type-II errors across different choices of sample size $n=m$ with fixed dimension $D=140$.
}
\label{fig:4}
\end{figure*}

Finally, we study the case where sample points are generated from high-dimensional Gaussian mixture distributions.
We take $\mu=\frac{1}{2}\mathcal{N}(0,I_D) + \frac{1}{2}\mathcal{N}(\Delta_2,I_D)$ with $\Delta_2=(1,1,\ldots,1)$ and $\nu = \frac{1}{2}\mathcal{N}(0,\Sigma_1) + \frac{1}{2}\mathcal{N}(\Delta_3,\Sigma_2)$ with $\Delta_3=(1+0.8/\sqrt{D},\ldots,1+0.8/\sqrt{D})$.
Covariance matrix $\Sigma_1$ is defined with $\Sigma_1[1,1]=\Sigma_1[2,2]=4, \Sigma_1[1,2]=\Sigma_1[2,1]=-0.9, \Sigma_1[i,i]=1, 3\le i\le D$, and $\Sigma_1[i,j]=0$ for indexes elsewhere.
Covariance matrix $\Sigma_2$ is defined with $\Sigma_2[1,2]=\Sigma_2[2,1]=0.9$, $\Sigma_2[i,i]=1, 1\le i\le D$, and $\Sigma_2[i,j]=0$ for indexes elsewhere.
Testing results (type-I and type-II errors) across different choices of dimension $D$ for fixed sample size $n=m=200$ is presented in the left two plots in Fig.~\ref{fig:4}.
We also report results for increasing sample sizes $n=m$ by fixing the dimension $D=140$ in the right two plots in Fig.~\ref{fig:4}.
From the plot, we can see that all approaches have expected type-I error rates.
Moreover, the tests based on PW and KPW distances outperform other benchmark methods, which indicates that the idea of dimension reduction is helpful for high-dimensional testing.
The KPW test generally has the highest power in this case, since the nonlinear projector in the unit ball of RKHS is flexible enough to capture the differences between distributions.
Other experiment details of this subsection is omitted in Appendix~\ref{Appendix:test:sync:dataset}.

\subsection{Tests for MNIST handwritten digits}
We now perform two-sample tests on the MNIST dataset~\citep{lecun10}.
Let $p$ be the distribution uniformly generated from the dataset, and $q=0.85 p + 0.15 p_{\text{cohort}}$, where $p_{\text{cohort}}$ is the distribution from a class with digit $1$.
Both training and testing sample sizes are set to be $N\in\{200,250,\ldots,500\}$.
Before performing two-sample tests, we pre-process this dataset by taking the sigmoid transformation of each image such that all scaled pixels are within the interval $[0,1]$.
Table~\ref{tab:MNIST_RES1} presents the testing power of various tests across different choices of $N$, from which we can see that the KPW test is competitive compared with other methods.
We observe that performances of MMD-O in MNIST dataset are significantly better than that in synthetic datasets provided in Section~\ref{Sec:Synetic}.
One possible explanation is that isotropic kernel functions will limit the power of MMD tests in some numerical examples~\citep[Section~3]{liu2020learning}.
Average type-I error for various tests is presented in Table~\ref{tab:MNIST_RES1:type:I} in Appendix~\ref{Appendix:test:mnist:detail}, from which we can see all tests have the type-I error close to $\alpha=0.05$.

\begin{table}[t!]
  \centering
  \footnotesize
  \caption{
  Delay time for detecting the transition in \emph{MSRC-12} that corresponds to four users.
  }
  \vspace{1mm}
    \begin{tabular}{cccccc}
    \toprule
  User &  MMD-NTK & MMD-O & ME & PW & KPW \\
    \midrule
    1 & 36 & 73 & 82 & 47 & {\bf 33} \\
    2 & 8 & 7 & 97 & 9 & {\bf 1} \\
    3 & 15 & 13 & 27 & {\bf 2} & 20 \\
    4 & 22 & 83 & 69 & 16 & {\bf 12} \\\midrule
    \text{Mean} & 20.25 & 44.0 & 68.8 & 18.50 & {\bf 16.5}\\
    \text{Std} & {\bf 12.0} & 39.5 & 30.1 & 19.8 & 13.5\\
    \bottomrule
    \end{tabular}%
  \label{tab:MSRC:full}%
  \vspace{-1em}
\end{table}%

\subsection{Human activity detection}
Finally, we apply the KPW test to perform online change-point detection for human activity transition.
We use a real-world dataset called the Microsoft Research Cambridge-12~(MSRC-12) Kinect gesture dataset~\citep{Fothergill12}. 
After pre-processing, this dataset consists of actions from four people, each with $855$ samples in $\mathbb{R}^{60}$, and with a change of action from \emph{bending} to \emph{throwing} at the time index $500$.
More experimental details are omitted in Appendix~\ref{Appendix:Human:data}.
Fix the window size $W=100$. We pre-train a nonlinear projector using the data (sample size as the window) before time index $300$ and compute the null statistics for many times to obtain the true threshold such that the false alarm rate is controlled within $\alpha=0.05$.
Then we perform online change-point detection based on a sliding window that moves forward with time.
We compute the detection statistic by comparing the distribution between the block of data before time $300$ and the data from the sliding window.
We reject the null hypothesis and claim a change is happened if the statistic is above the threshold.
Table~\ref{tab:MSRC:full} reports the delay time for detecting the behavior transition, from which we observe that the KPW test detects the change in the shortest time.

\section{
CONCLUSION
}\label{Sec:discussion}
We proposed the KPW distance for the task of two-sample testing, which operates by finding the nonlinear mapping in the data space to maximize the distance between projected distributions.
Practical algorithms together with uncertainty quantification of empirical estimates are discussed to help with this task.

The extension of this work is as follows.
First, it is promising to consider milder technical assumptions than the projected Poincare inequality when establishing performance guarantees.
Second, a meaningful research question is to determine the optimal hyper-parameters for the KPW test, including the projected subspace dimension $d$ and the matrix-valued kernel function $K$.
Third, it is desirable to study how to systematically pick the regularization parameter $\eta$ to balance the trade-off between computational efficiency and accuracy of the obtained solution.

\section*{Acknowledgements}
This work is supported by NSF DMS-2134037, CCF-1650913, CMMI-2015787, DMS-1938106, and DMS-1830210.
The authors would like to thank the
Editor and the anonymous referees for the thoughtful
comments and suggestions, which led to an improvement of the presentation.

\bibliography{sample}
\bibliographystyle{apalike}
\newpage

\newpage

\if\submitversion1
\appendix
\thispagestyle{empty}
\onecolumn
\makesupplementtitle
\fi
\if\submitversion2
\appendices
\onecolumn
\fi

\section{
PRELIMINARY TECHNICAL RESULTS
}
\begin{theorem}[Pinsker's Inequality~\citep{Cover06}]\label{Theorem:Pinsker}
Consider two discrete probability distributions $p=\{p_i\}_{i=1}^n$ and $q=\{q_i\}_{i=1}^n$, then it holds that
\[
\sum_{i=1}^np_i\log\frac{p_i}{q_i}\ge \frac{1}{2}\|p-q\|_1^2.
\]
\end{theorem}

\begin{proposition}[Lipschitz Properties of Retraction Operator~\citep{Boumal_2018}]\label{Proposition:Retraction:Lipschitz}
There exists constants $L_1,L_2$ such that the following inequalities hold:
\begin{align*}
    \|\text{Retr}_{s}(\zeta) - s\|&\le L_1\|\zeta\|\\
    \|\text{Retr}_{s}(\zeta) - (s+\zeta)\|&\le L_2\|\zeta\|^2.
\end{align*}
\end{proposition}
Inspired from Appendix~A.3 in \cite{Jiang2017VectorTS}, we are able to compute the constants in Proposition~\ref{Proposition:Retraction:Lipschitz} explicitly: $L_1=1$ and $L_2=\frac{1}{2}$.
The proof is provided below.
\begin{proof}
By definition, we have that 
\begin{align*}
    \|\text{Retr}_{s}(\zeta) - s\|_2^2&=\left\|
    \frac{s+\zeta}{\|s+\zeta\|} - s
    \right\|_2^2\\
    &=2\left(
    1 - \frac{1}{\|s+\zeta\|_2}
    \right)\\
    &=2\left( 
    1 - (1 + \sum_i\zeta_i^2)^{-1/2}
    \right)\\
    &\le \sum_i\zeta_i^2 = \|\zeta\|_2^2.
\end{align*}
where the second and the third equality is by using the relation $s\trans\zeta=0$, and the inequality is based on the relation $2(1 - (1+z)^{-1/2})\le z$ with $z=\sum_i\zeta_i^2$.
Then it holds that $\|\text{Retr}_{s}(\zeta) - (s+\zeta)\|_2\le \|\zeta\|$.

Secondly, we can see that 
\begin{align*}
\|\text{Retr}_{s}(\zeta) - (s+\zeta)\|_2^2
&=\left\|
    \frac{s+\zeta}{\|s+\zeta\|} - (s+\zeta)
    \right\|_2^2\\
&=\left( 
1 - \|s+\zeta\|_2
\right)^2\\
&=\left(
1 - \sqrt{1 + \sum_i\zeta_i^2}
\right)^2\\
&\le \frac{1}{4}\|\zeta\|_2^4,
\end{align*}
where the inequality is based on the relation that $(1 - (1+z)^{1/2})^2\le z^2/4$ with $z=\sum_i\zeta_i^2$.
Consequently it holds that $\|\text{Retr}_{s}(\zeta) - (s+\zeta)\|_2\le\frac{1}{2}\|\zeta\|^2$.
\end{proof}

\begin{theorem}[McDiarmid's Inequality~\citep{mcdiarmid_1989}]
\label{The:Mcdiarmid}
Let $X_1,\ldots,X_n$ be independent random variables, where $X_i$ has the support $\mathcal{X}_i$.
Let $f:\mathcal{X}_1\times\mathcal{X}_2\times\cdots\times\mathcal{X}_n\to\mathbb{R}$ be any function
with the $(c_1,\ldots,c_n)$ bounded difference property, i.e., for $i\in\{1,\ldots,n\}$ and for any $(x_1,\ldots,x_n), (x_1',\ldots,x_n')$ that  differs only in the $i$-th corodinate, we have
\[
|f(x_1,\ldots,x_n) - f(x_1',\ldots,x_n')|\le c_i.
\]
Then for any $t>0$, we have
\[
\text{Pr}\bigg\{
|f(X_1,\ldots,X_n)
-
\mathbb{E}[f(X_1,\ldots,X_n)]
|
\ge t
\bigg\}
\le
2\exp\left(
-\frac{2t^2}{\sum_{i=1}^nc_i^2}
\right).
\]
\end{theorem}
\begin{lemma}[Equivalent Definition for Sub-Gaussian variables (Lemma~2.3.2 in \citep{Gin2015})]\label{Lemma:equivalent:subGaussian}
Assume that $\mathbb{E}[\zeta]=0$ and 
\[
\mathbb{P}\{|\zeta|\ge t\}\le 2C\exp\left(
-\frac{t^2}{2\sigma^2}
\right),\quad t>0,
\]
for some $C\ge 1$ and $\sigma>0$.
Then the random variable $\zeta$ is sub-Gaussian with constant
$\tilde{\sigma}^2=12(2C+1)\sigma^2$.
\end{lemma}

\begin{theorem}[Poincare's Inequality]\label{Theorem:poincare}
Denote by $\mu^n$ the product of $\mu$ on $\otimes_{i=1}^n\mathbb{R}^d$ and $\mu\in\mathcal{P}(\mathbb{R}^d)$ satisfies the Poincare's inequality, i.e., there exists $M>0$ for $X\sim \mu$ so that $\text{Var}[f(X)] \le M\mathbb{E}[\|\nabla f(X)\|^2]$ for any $f$ satisfying $\mathbb{E}[f(X)^2]<\infty$ and $\mathbb{E}[\|\nabla f(X)\|_2^2]<\infty$.
Consider a function $f$ on $\otimes_{i=1}^n\mathbb{R}^d$ satisfying $\mathbb{E}|f(X)|<\infty$ and $\sum_{i=1}^n\|\nabla_if(X)\|^2\le\alpha^2$, and $\max_{1\le i\le n}\|\nabla_if(X)\|\le \beta$ almost surely.
Then the following inequality holds for $X\sim\mu^n$:
\[
\text{Pr}\bigg\{
f(X)
-
\mathbb{E}[f(X)]
>t
\bigg\}
\le
\exp\left(
-\frac{1}{K}\min(t/\beta, t^2/\alpha^2)
\right).
\]
\end{theorem}

\clearpage
\section{
INTRODUCTION TO MANIFOLD OPTIMIZATION
}\label{Appendix:manifold}
A brief introduction to manifold optimization can be found in \cite{hu2019brief}.
In this section we list some related operators for solving manifold optimization problems.
Traditional manifold optimization concerns with solving the following problem:
\begin{equation}
    \min_{x\in\mathcal{M}}~f(x),
\end{equation}
where $\mathcal{M}$ is a Riemannian manifold and $f$ is a real-valued function on $\mathcal{M}$.
A tangent vector $\zeta_x$ to $\mathcal{M}$ at a point $x$ is defined as a mapping so that 
there exists a curve $\gamma$ on $\mathcal{M}$ satisfying
\[
\gamma(0)=x,\quad 
\zeta_x[u]=\frac{\diff(u(\gamma(t)))}{\diff t}\mid_{t=0},\ 
\forall u\in\mathfrak{E}(\mathcal{M}),
\]
where $\mathfrak{E}(\mathcal{M})$ stands for the collection of real-valued functions defined in a neighborhood of $x$.
Denote by $T_x\mathcal{M}$ as the collection of all tangent vectors to $\mathcal{M}$ at a point $x$, which is called the tangent space to $\mathcal{M}$ at $x$.
Define $\mathcal{P}_x(z)$ as the projection of $z$ into the tangent space at $x$.
Based on definitions listed above, we can define necessary operators for manifold optimization.
The Riemannian gradient of $f$ at $x$ is denoted as $\text{Grad}f(x)$, which can be obtained by projecting the gradient of $f$ at $x$ in the Euclidean space into the tangent space to $\mathcal{M}$ at $x$:
\[
\text{Grad}f(x) = \mathcal{P}_{x}(\nabla f(x)).
\]
Typical Riemannian manifolds include the Sphere and Stiefel manifold defined as follows:
\begin{align*}
\text{Sphere}(n-1)&:=\{
x\in\mathbb{R}^n:~\|x\|_2=1
\},\\
\text{St}(n,p)&:=\{
X\in\mathbb{R}^{n\times p}:~
X\trans X=I_p
\}.
\end{align*}
We can express the tangent space together with the projection operator for these two types of manifolds in analytical form:
\begin{align*}
T_x\text{Sphere}(n-1)&=\{z:~z\trans x=0\},\quad 
\mathcal{P}_x(z)
=(I-xx\trans)z\\
T_x\text{St}(n,p)&=\{Z:~Z\trans X + X\trans Z=0\},\quad
\mathcal{P}_X(Z)=Z-X\frac{X\trans Z + Z\trans X}{2}.
\end{align*}
When using first-order methods to solve a manifold optimization problem, one also needs to define the retraction operator associated with $\mathcal{M}$, which is denoted as $\text{Retr}$.
It is a smooth mapping from the tangent budle $\cup_{x\in\mathcal{M}}T_x\mathcal{M}$ to $\mathcal{M}$ satisfying that for any $x\in\mathcal{M}$,
\begin{itemize}
    \item 
$\text{Retr}_{x}(0_x) = x$, where $0_x$ denotes the zero element in $T_x\mathcal{M}$;
    \item
$\lim_{\zeta\in T_x\mathcal{M}, \zeta\to0}\frac{\|\text{Retr}_x(\zeta) - (x+\zeta)\|}{\|\zeta\|}=0$.
\end{itemize}
When $\mathcal{M}$ is a sphere, we choose the following retraction operator which can be implemented efficiently:
\[
\text{Retr}_{x}\big(\zeta\big) = \frac{x+\zeta}{\|x+\zeta\|},\ \quad x\in\text{Sphere}(n-1).
\]
See \cite{edelman1998geometry} and \cite{Wen13} for discussions of retraction operators on the Stiefel manifold.
The general iteration update of first-order methods for manifold optimization problem can be expressed as 
\[
x^{t+1} = \text{Retr}_{x^t}(-\tau^{t}\zeta^{t}),
\]
where $\tau^t$ is a well-defined step size and $\zeta^t$ is the Riemannian gradient at $x^t$.
The computation of the projected Wasserstein distance relates to the optimization on a Stiefel manifold, while the computation of the KPW distance relates to the optimization on a sphere.
A recent paper \citep{Boumal_2018} investigated the Riemannian gradient methods that are guaranteed to converge into stationary points globally, the key proof technique of which relies on Proposition~\ref{Proposition:Retraction:Lipschitz}.
We follow the similar proof idea to establish the convergence analysis for computing the KPW distance.

\section{
TECHNICAL PROOFS IN SECTION~\ref{Sec:setup}
}
\label{app:theorem:sec:setup}

\begin{proof}[Proof of Remark~\ref{Remark:KPW:choose:kernel}]
When taking the kernel function $K(x,y)=\inp{x}{y}$, the space
\[
\mathcal{F} = \{a:~a\trans a\le 1\}.
\]
Note that the cost function $c(x,y)=\|x-y\|_2^2$ satisfies $c(mx,my)=m^2c(x,y)$ for any $m\in\mathbb{R}$.
Hence we can argue that the maximizer of the KPW distance is obtained when $a\trans a=1$, i.e.,
\[
\begin{aligned}
\mathcal{KP}W(\mu, \nu)
&=
\max_{\substack{f:~\mathbb{R}^D\to\mathbb{R}, \\
f(z)=a\trans z, a\trans a=1}}
~W\left(
f\#\mu,
f\#\nu
\right).
\end{aligned}
\]
This indicates that the KPW distance reduces into the PW distance.
\end{proof}

\begin{proof}[Proof of Proposition~\ref{Proposition:discriminant}]
It is easy to see that $\mu=\nu$ implies $\mathcal{KP}W(\mu,\nu)=0$.
Now we show the converse.
For fixed $x\in\mathcal{X}, y\in\mathbb{R}^d$ and a distribution $\mu$, define the operator $K_{\mu}$ with the action $y$ as a mapping $K_{\mu}y:~\mathcal{X}\to\mathbb{R}^d$ so that 
\[
K_{\mu}y(x') = \int (K_xy)(x')\diff\mu(x)=\int K(x',x)y\diff\mu(x).
\]
When $\mathcal{KP}W(\mu,\nu)=0$, we can see that
\[
f\#\mu = f\#\nu,\quad\forall f\in\mathcal{F},
\]
which implies
\begin{align*}
0&=
\sup_{f:~\|f\|_{\mathcal{H}}^2\le 1}~\big\|
\mathbb{E}_{f\#\mu}[x] - \mathbb{E}_{f\#\nu}[y]
\big\|_2\\
&=\sup_{f:~\|f\|_{\mathcal{H}}^2\le 1}\sup_{a:~\|a\|_2\le 1}~\big(
\mathbb{E}_{\mu}[\inp{f(x)}{a}] - \mathbb{E}_{\nu}[\inp{f(y)}{a}]
\big)\\
&=
\sup_{f:~\|f\|_{\mathcal{H}}^2\le 1}\sup_{a:~\|a\|_2\le 1}~\big(
\mathbb{E}_{\mu}[\inp{f}{K_xa}_{\mathcal{H}}] - \mathbb{E}_{\nu}[\inp{f}{K_ya}_{\mathcal{H}}]
\big)\\
&=\sup_{f:~\|f\|_{\mathcal{H}}^2\le 1}\sup_{a:~\|a\|_2\le 1}~
\inp{f}{(K_{\mu}- K_{\nu})a}\\
&=\sup_{a:~\|a\|_2\le 1}~\|(K_{\mu}- K_{\nu})a\|_{\mathcal{H}}.
\end{align*}
Equivalently, $\|(K_{\mu}- K_{\nu})a\|_{\mathcal{H}}=0$ for any $a$ so that $\|a\|_2\le 1$.
Since $\mathcal{H}$ is a Hilbert space, we imply that $(K_{\mu}- K_{\nu})a$ is a zero function for any $a$ satisfying $\|a\|_2\le 1$.
For any function $f\in\mathcal{C}(X)$, we make the expansion
\begin{align*}
&\left\|
\mathbb{E}_{\mu}[f(x)]
-
\mathbb{E}_{\nu}[f(y)]
\right\|_2
\\\le& 
\left\|
\mathbb{E}_{\mu}[f(x)]
-
\mathbb{E}_{\mu}[g(x)]
\right\|_2
+
\left\|
\mathbb{E}_{\mu}[g(x)]
-
\mathbb{E}_{\nu}[g(y)]
\right\|_2
+
\left\|
\mathbb{E}_{\nu}[g(y)]
-
\mathbb{E}_{\nu}[f(y)]
\right\|_2.
\end{align*}
The first term satisfies that
\[
\left\|
\mathbb{E}_{\mu}[f(x)]
-
\mathbb{E}_{\mu}[g(x)]
\right\|_2
\le 
\mathbb{E}_{\mu}[
\|f(x) - g(x)\|_2
]<\varepsilon,
\]
and the third term can be upper bounded likewise.
For the second term, we have that
\begin{align*}
&\left\|
\mathbb{E}_{\mu}[g(x)]
-
\mathbb{E}_{\nu}[g(y)]
\right\|_2\\
=&
\sup_{a:~\|a\|_2\le 1}~
\big(
\mathbb{E}_{\mu}[\inp{g(x)}{a}]
-
\mathbb{E}_{\nu}[\inp{g(y)}{a}]
\big)\\
=&\sup_{a:~\|a\|_2\le 1}~
\big(
\mathbb{E}_{\mu}[\inp{g}{K_xa}]
-
\mathbb{E}_{\nu}[\inp{g}{K_ya}]
\big)\\
=&\sup_{a:~\|a\|_2\le 1}~
\inp{g}{(K_{\mu}-K_{\nu})a}=0,
\end{align*}
where the last equality is because that $(K_{\mu}-K_{\nu})a$ is a zero function for any $a$ satisfying $\|a\|_2\le 1$.
Hence, $\left\|
\mathbb{E}_{\mu}[f(x)]
-
\mathbb{E}_{\nu}[f(y)]
\right\|_2<2\varepsilon$ for any $\varepsilon>0$ and $f\in\mathcal{C}_b(\mathcal{X})$.
Then we conclude that the distribution $\mu=\nu$.
\end{proof}

\clearpage
\section{
TECHNICAL PROOFS IN SECTION~\ref{Sec:algorithm}
}\label{Appendix:proof:algorithm}

\subsection{Deviation of Duality Reformulation~(\ref{Eq:final:min:three:block})}
We first present the proof of the dual reformulation of the inner minimization problem in \eqref{Eq:max:min:entropy}.
By definition, the primal formulation can be expressed as:
\begin{equation}
\min_{\pi\ge0}
~\left\{
\sum_{i,j}\pi_{i,j}
c_{i,j}
-\eta \sum_{i,j}\pi_{i,j}(\log\pi_{i,j}-1):\quad 
\sum_j\pi_{i,j}=\frac{1}{n}, \sum_i\pi_{i,j}=\frac{1}{m}
\right\}.
\end{equation}
The Lagrangian function becomes
\[
L(\pi, u, v) = \sum_{i,j}\pi_{i,j}
c_{i,j}
-\eta \sum_{i,j}\pi_{i,j}(\log\pi_{i,j}-1) + \sum_iu_i\left( 
\sum_j\pi_{i,j} - \frac{1}{n}
\right) + \sum_jv_j\left( 
\sum_i\pi_{i,j}-\frac{1}{m}
\right)
\]
Then the dual problem becomes 
\begin{align*}
&\max_{u, v}~\left\{\min_{\pi\ge0}~L(\pi, u, v)\right\}\\
=&\max_{u, v}~-\frac{1}{n}\sum_iu_i - \frac{1}{m}\sum_jv_j + 
\min_{\pi\ge0}~\sum_{i,j}\pi_{i,j}\big[c_{i,j} + u_i + v_j\big] - \eta\pi_{i,j}(\log\pi_{i,j}-1)\\
=&\max_{u, v}~-\frac{1}{n}\sum_iu_i - \frac{1}{m}\sum_jv_j - \sum_{i,j}
\max_{\pi_{i,j}\ge0}~\left\{ 
-\pi_{i,j}\big[c_{i,j} + u_i + v_j\big] + \eta\pi_{i,j}(\log\pi_{i,j}-1)
\right\}\\
=&\max_{u, v}~-\frac{1}{n}\sum_iu_i - \frac{1}{m}\sum_jv_j - \sum_{i,j}(\eta\phi)^*(u_i+v_j + c_{i,j})\\
=&\max_{u, v}~-\frac{1}{n}\sum_iu_i - \frac{1}{m}\sum_jv_j - \eta\sum_{i,j}\exp\left( 
-\frac{u_i+v_j + c_{i,j}}{\eta}
\right)
\end{align*}
where $\phi(w)=w\log w - w$ and $\phi^*$ denotes its conjugate~\citep{Rockafellar70}.
Moreover, the dual optimal value equals the primal optimal value because the Slater's condition~\citep{boyd2004convex} for finite-dimensional optimization is satisfied.
Take $u_i' = -u_i/\eta$ and $v_j' = -v_j/\eta$, the dual problem becomes
\[
\max_{u', v'}~\frac{\eta}{n}\sum_iu_i' + \frac{\eta}{m}\sum_jv_j' - \eta\sum_{i,j}\exp\left( 
-\frac{c_{i,j}}{\eta} + u_i' + v_j'
\right).
\]
Therefore, the whole problem \eqref{Eq:max:min:entropy} becomes 
\[
\max_{u,v,s}~\frac{\eta}{n}\sum_iu_i + \frac{\eta}{m}\sum_jv_j - \eta\sum_{i,j}\exp\left( 
-\frac{c_{i,j}}{\eta} + u_i + v_j
\right).
\]
Or equivalently, we write it as the minimization problem:
\[
-\eta\times\left\{\min_{u,v,s}~-\frac{1}{n}\sum_iu_i - \frac{1}{m}\sum_jv_j
+\eta\sum_{i,j}\exp\left( 
-\frac{c_{i,j}}{\eta} + u_i + v_j
\right)\right\}.
\]

\begin{remark}
By adding the entropic regularization term $\eta H(\pi)$, we are able to derive an unconstrained optimization formulation on the sphere, thus reducing the computational cost for computing KPW distance. 
Besides, the induced optimal transport mapping between projected samples is usually stochastic instead of deterministic,
which is robust to potential data outliers.
\end{remark}

\subsection{Proof of Theorem~\ref{Theorem:representer}}
Assume that $\hat{f}$ is an optimal solution to the problem \eqref{Eq:nominal:KPW}.
Let $S$ be the subspace
\[
S=\left\{
\sum_{i=1}^n\sum_{j=1}^m (K_{x_i}-K_{y_j})a_{i,j}:~
a_{i,j}\in\mathbb{R}^d
\right\}.
\]
Denote by $S_{\perp}$ the orthogonal complement of $S$.
Given a set $\mathcal{X}$, denote by $f_{\mathcal{X}}$ a function that lies in the set $\mathcal{X}$.
Then by the projection theorem, there exists $\hat{f}_{S}$ and $\hat{f}_{S_{\perp}}$ such that 
$\hat{f} = \hat{f}_{S} + \hat{f}_{S_{\perp}}$ and $\|\hat{f}\|_{\mathcal{H}}^2=\|\hat{f}_{S}\|_{\mathcal{H}}^2+\|\hat{f}_{S_{\perp}}\|_{\mathcal{H}}^2$. 
It remains to show that $\hat{f}_{S}$ shares the same objective value with $\hat{f}$.
For fixed $i,j$, we have that
\begin{align*}
\|\hat{f}(x_i) - \hat{f}(y_j)\|_2&=
\max_{a_{i,j}:~\|a_{i,j}\|_2\le 1}
\inp{\hat{f}(x_i)-\hat{f}(y_j)}{a_{i,j}}
\\
&=
\max_{a_{i,j}:~\|a_{i,j}\|_2\le 1}
\inp{\hat{f}(x_i)}{a_{i,j}} - \inp{\hat{f}(y_j)}{a_{i,j}}
\\
&=
\max_{a_{i,j}:~\|a_{i,j}\|_2\le 1}
\inp{\hat{f}}{K_{x_i}a_{i,j}} - \inp{\hat{f}}{K_{y_j}a_{i,j}}
\\
&=
\max_{a_{i,j}:~\|a_{i,j}\|_2\le 1}
\inp{\hat{f}}{(K_{x_i}-K_{y_j})a_{i,j}} \\
&=
\max_{a_{i,j}:~\|a_{i,j}\|_2\le 1}
\inp{\hat{f}_S}{(K_{x_i}-K_{y_j})a_{i,j}}
=\|\hat{f}_S(x_i) - \hat{f}_S(y_j)\|_2,
\end{align*}
where the second last equality is because $\hat{f}_{S_{\perp}}$ is orthogonal to the subspace $S$.
It follows that $\|\hat{f}(x_i) - \hat{f}(y_j)\|_2^2 = \|\hat{f}_S(x_i) - \hat{f}_S(y_j)\|_2^2$.
Therefore, there always exists an optimal solution that lies in the subspace $S$, which means that there exists an optimal solution to \eqref{Eq:nominal:KPW} that admits the following expression:
\[
\hat{f} = \sum_{i=1}^n\sum_{j=1}^m (K_{x_i}-K_{y_j})a_{i,j}.
\]
Defining $a_{x,i}=\sum_{j=1}^ma_{i,j}$ and $a_{y,j}=\sum_{i=1}^na_{i,j}$ completes the proof.

\begin{remark}
From the proof we can also see that the representer theorem holds if replacing the square of the $\ell_2$ norm in \eqref{Eq:nominal:KPW} with any $p$-th power of the $\ell_2$ norm for $p\ge2$.
However, we find the development of optimization algorithms for the square of the $\ell_2$ norm case is the simplest.
\end{remark}

\subsection{Proof of Theorem~\ref{Theorem:convergence:BCD}}
In the following we give a iteration complexity analysis about Algorithm~\ref{Alg:1}, the proof of which largely follows the idea in \cite{huang2021riemannian}.
In particular, we first establish the descent lemma for the update of each block of variables and then argue that the objective function is lower bounded.
Based on these two facts, we finally build the iteration complexity result for Algorithm~\ref{Alg:1}.

\begin{lemma}[Lipschitzness of $\nabla_{s}F(u,v,s)$]\label{Lemma:Lipschitz}
Let $\{u^t,v^t,s^t\}_{t}$ be the sequence generated from Algorithm~\ref{Alg:1}.
The following inequality holds for any $s\in\mathbb{S}^{d(n+m)-1}$ and $\lambda\in[0,1]$:
\[
\|
\nabla_{s}F(u^{t+1},v^{t+1},\lambda s+(1-\lambda)s^t)
-
\nabla_{s}F(u^{t+1},v^{t+1},s^t)
\|
\le 
\varrho \lambda\|s^t-s\|,
\]
where $\varrho=\frac{2\|AU\|_\infty^2}{\eta} + \frac{4\|AU\|_\infty^4}{\eta^2}$ and $\|AU\|_{\infty} = \max_{i,j}\|A_{i,j}U\|_2$.
\end{lemma}
\begin{proof}[Proof of Lemma~\ref{Lemma:Lipschitz}]
An intermediate result is that 
\begin{align*}
\sum_i{\pi_{i,j}}(u^{t+1}, v^{t+1}, s^t)
&=
\sum_i\exp\left(
-\frac{1}{\eta}c_{i,j}[s^t]+u_i^{t+1}
\right)\exp\left(
v^{t+1}_j
\right)\\
&=
\sum_i\exp\left(
-\frac{1}{\eta}c_{i,j}[s^t]+u_i^{t+1}
\right)\exp\left(
v^{t}_j
\right)\frac{1/m}{\sum_i{\pi_{i,j}}(u^{t+1}, v^t, s^t)}\\
&=
\frac{1}{m}\frac{\sum_i{\pi_{i,j}}(u^{t+1},v^{t},s^{t})}{\sum_i{\pi_{i,j}}(u^{t+1}, v^t, s^t)}=1/m.
\end{align*}
Then we can assert that $\sum_{i,j}{\pi_{i,j}}(u^{t+1}, v^t, s^t) = 1$.
For fixed $s^t$, define $s^{\lambda} = \lambda s+(1-\lambda)s^t$. 
Then we have that
\begin{align*}
&\|
\nabla_{s}F(u^{t+1},v^{t+1},s^{t})
-
\nabla_{s}F(u^{t+1},v^{t+1},s^{\lambda})
\|\\
=&\frac{2}{\eta}
\left\|
\sum_{i,j}\pi_{i,j}(u^{t+1},v^{t+1},s^t) U\trans A_{i,j}\trans A_{i,j}Us^{t}
-
\sum_{i,j}\pi_{i,j}(u^{t+1},v^{t+1},s^{\lambda}) U\trans A_{i,j}\trans A_{i,j}Us^{\lambda}
\right\|\\
\le &\frac{2}{\eta}
\left\|
\sum_{i,j}\pi_{i,j}(u^{t+1},v^{t+1},s^t) U\trans A_{i,j}\trans A_{i,j}U (s^{t} - s^{\lambda})
\right\|
\\
&\quad
+
\frac{2}{\eta}
\left\|
\sum_{i,j}U\trans\big[ \pi_{i,j}(u^{t+1},v^{t+1},s^t) - \pi_{i,j}(u^{t+1},v^{t+1},s^{\lambda})\big]A_{i,j}\trans A_{i,j}U
\right\|
\\
\le&
\frac{2}{\eta}\left\|
\sum_{i,j}\pi_{i,j}(u^{t+1},v^{t+1},s^t) U\trans A_{i,j}\trans A_{i,j}U
\right\|
\|s^{\lambda} - s^t\|
\\
&\quad+\frac{2}{\eta}
\left\|
\sum_{i,j}\big[ \pi_{i,j}(u^{t+1},v^{t+1},s^t) - \pi_{i,j}(u^{t+1},v^{t+1},s^{\lambda})\big]U\trans A_{i,j}\trans A_{i,j}U
\right\|
\end{align*}
where the first inequality is based on the constraint that $\|s^{\lambda}\|\le \lambda\|s\| + (1-\lambda)\|s^t\|=1$.
To upper bound the first term, we find
\begin{align*}
&\left\|
\sum_{i,j}\pi_{i,j}(u^{t+1},v^{t+1},s^t) U\trans A_{i,j}\trans A_{i,j}U
\right\|\\
\le&\sum_{i,j}\pi_{i,j}(u^{t+1},v^{t+1},s^t)\|U\trans A_{i,j}\trans A_{i,j}U\|_2
\le\max_{i,j}\|A_{i,j}U\|_2^2.
\end{align*}
To bound the second term, we find that 
\begin{align*}
&\left\|
\sum_{i,j}\big[ \pi_{i,j}(u^{t+1},v^{t+1},s^t) - \pi_{i,j}(u^{t+1},v^{t+1},s^{\lambda})\big]U\trans A_{i,j}\trans A_{i,j}U
\right\|\\
\le&\max_{i,j}\|A_{i,j}U\|_2^2\|\pi(u^{t+1}, v^{t+1}, s^{\lambda}) -  \pi(u^{t+1}, v^{t+1}, s^t)\|_1,
\end{align*}
where
\[
\|\pi(u^{t+1}, v^{t+1}, s^{\lambda}) -  \pi(u^{t+1}, v^{t+1}, s^t)\|_1:=
\sum_{i,j}\big|
\pi_{i,j}(u^{t+1}, v^{t+1}, s^{\lambda}) - \pi_{i,j}(u^{t+1}, v^{t+1}, s^t)
\big|.
\]
Denote by $H(\pi, s; \eta)$ the objective function for \eqref{Eq:finite}.
Based on the strong convexity property, we have that 
\begin{align*}
&\inp{\nabla_{\pi} H(\pi(u^{t+1}, v^{t+1}, s^{\lambda}), s^{\lambda}; \eta) - \nabla_{\pi} H(\pi(u^{t+1}, v^{t+1}, s^{t}), s^{\lambda}; \eta)}{\pi(u^{t+1}, v^{t+1}, s^{\lambda}) - \pi(u^{t+1}, v^{t+1}, s^t)}\\
\ge&\eta \|\pi(u^{t+1}, v^{t+1}, s^{\lambda}) -  \pi(u^{t+1}, v^{t+1}, s^t)\|_1^2
\end{align*}
Moreover, by simple calculation we find
\begin{align*}
\nabla_{\pi} H(\pi(u,v,s), s)&=\left[ c_{i,j}+\eta\log(\pi_{i,j}(u,v,s)) \right]_{i,j}\\
&=\left[ \eta(u_i + v_j) \right]_{i,j},
\end{align*}
where the second equality is by substituting the formulation of $\pi_{i,j}(u,v,s)$.
Hence, we find that the gradient $\nabla_{\pi} H(\pi(u,v,s), s)$ only depends on $u$ and $v$,
which implies
\begin{align*}
&\inp{\nabla_{\pi} H(\pi(u^{t+1}, v^{t+1}, s^t), s^t; \eta) - \nabla_{\pi} H(\pi(u^{t+1}, v^{t+1}, s^{t}), s^{\lambda}; \eta)}{\pi(u^{t+1}, v^{t+1}, s^{\lambda}) - \pi(u^{t+1}, v^{t+1}, s^t)}\\
\ge&\eta \|\pi(u^{t+1}, v^{t+1}, s^{\lambda}) -  \pi(u^{t+1}, v^{t+1}, s^t)\|_1^2.
\end{align*}
It follows that
\begin{align*}
&\eta\|\pi(u^{t+1}, v^{t+1}, s^{\lambda}) -  \pi(u^{t+1}, v^{t+1}, s^t)\|_1\\
\le&\|\nabla_{\pi} H(\pi(u^{t+1}, v^{t+1}, s^t), s^t; \eta) - \nabla_{\pi} H(\pi(u^{t+1}, v^{t+1}, s^{t}), s^{\lambda}; \eta)\|_{\infty}\\
=&\max_{i,j}\big| 
\|A_{i,j}Us^{\lambda}\|_2^2 - \|A_{i,j}Us^{t}\|_2^2
\big|\\
\le&2\max_{i,j}\|A_{i,j}U\|_2^2\|s^{\lambda} - s^t\|.
\end{align*}
where the inequality is by applying the following relation:
\begin{align*}
\|Ax_1\|_2^2 - \|Ax_2\|_2^2
&=(x_1-x_2)\trans (A\trans Ax_1) + x_2\trans A\trans A(x_1-x_2)\\
&\le \|x_1-x_2\|\|A\trans Ax_1\| + \|x_2\trans A\trans A\|\|x_1-x_2\|\\
&\le 2\|A\|^2\|x_1-x_2\|.
\end{align*}
In summary, the second term can be upper bounded as 
\begin{align*}
&\left\|
\sum_{i,j}\big[ \pi_{i,j}(u^{t+1},v^{t+1},s^t) - \pi_{i,j}(u^{t+1},v^{t+1},s^{\lambda})\big]U\trans A_{i,j}\trans A_{i,j}U
\right\|\\
\le&\frac{2\left(\max_{i,j}\|A_{i,j}U\|_2^2\right)^2}{\eta}\|s^{\lambda} - s^t\|.
\end{align*}
Then applying the condition that 
$\|s^{\lambda} - s^{t}\|=\lambda\|s-s^t\|$ completes the proof.
\end{proof}

\begin{lemma}[Decrease of $F$ in $s$]\label{Lemma:decrease:F:s}
Let $\{u^t,v^t,s^t\}_{t}$ be the sequence generated from Algorithm~\ref{Alg:1}.
The following inequality holds for any $k\ge1$:
\[
F(u^{t+1},v^{t+1},s^{t+1}) - F(u^{t+1},v^{t+1},s^{t})
\le
-\frac{1}{8\|AU\|_\infty^2L_2/\eta+2\varrho L_1^2}
\|\xi^{t+1}\|^2.
\]
\end{lemma}

\begin{proof}[Proof of Lemma~\ref{Lemma:decrease:F:s}]
Note that
\begin{align*}
&\left|
F(u^{t+1},v^{t+1},s^{t+1}) - F(u^{t+1},v^{t+1},s^{t})
-
\inp{\nabla_{t}F(u^{t+1},v^{t+1},s^{t})}{s^{t+1}-s^t}
\right|\\
=&
\left|
\int_{0}^{1}
\inp{
\nabla_{s}F(u^{t+1},v^{t+1},\lambda s^{t+1}+(1-\lambda)s^t)
-
\nabla_{s}F(u^{t+1},v^{t+1},s^t)
}{s^{t+1}-s^t}\diff\lambda
\right|\\
\le&
\int_{0}^{1}
\|
\nabla_{s}F(u^{t+1},v^{t+1},\lambda s^{t+1}+(1-\lambda)s^t)
-
\nabla_{s}F(u^{t+1},v^{t+1},s^t)
\|
\|s^{t+1}-s^t\|\diff\lambda\\
\le&
\int_0^1\varrho\lambda\|s^{t+1} - s^{t}\|^2\diff\lambda\\
=&\frac{\varrho}{2}\|s^{t+1} - s^{t}\|^2=
\frac{\varrho}{2}\left\|
\text{Retr}_{s^t}\big(-\tau \xi^{t+1}\big) - s^t
\right\|^2\\
\le &\frac{\varrho\tau^2L_1^2}{2}\|\xi^{t+1}\|^2.
\end{align*}
where the second inequality is by applying Lemma~\ref{Lemma:Lipschitz},
and the last inequality is by applying Proposition~\ref{Proposition:Retraction:Lipschitz}.
Moreover, we have that 
\begin{align*}
&\inp{\nabla_{s}F(u^{t+1},v^{t+1},s^{t})}{s^{t+1}-s^t}\\
=&\inp{\nabla_{s}F(u^{t+1},v^{t+1},s^{t})}{-\tau\xi^{t+1}} + 
\inp{\nabla_{s}F(u^{t+1},v^{t+1},s^{t})}{\text{Retr}_{s^t}\big(-\tau \xi^{t+1}\big) - (s^t - \tau\xi^{t+1})}
\\
\le &
-\tau\|\xi^{t+1}\|^2 + \|\nabla_{s}F(u^{t+1},v^{t+1},s^{t})\|_2\|\text{Retr}_{s^t}\big(-\tau \xi^{t+1}\big) - (s^t - \tau\xi^{t+1})\|\\
\le &
-\tau\|\xi^{t+1}\|^2 + \|\zeta^{t+1}\|_2\cdot L_2\tau^2\|\xi^{t+1}\|^2\\
\le& -\tau\|\xi^{t+1}\|^2 + \frac{2\|AU\|_\infty^2L_2\tau^2}{\eta}
\|\xi^{t+1}\|^2.
\end{align*}
Combining those inequalities above implies that
\[
F(u^{k+1},v^{k+1},t^{k+1}) - F(u^{k+1},v^{k+1},t^{k})
\le 
-\tau\left(
1 -\left[ 
\frac{2\|AU\|_\infty^2L_2}{\eta}
+
\frac{\varrho}{2}L_1^2
\right]\tau 
\right)\|\xi^{t+1}\|^2.
\]
Taking $\tau=\frac{1}{4\|AU\|_\infty^2L_2/\eta+\varrho L_1^2}$ gives the desired result.
\end{proof}

\begin{lemma}[Decrease of $F$ in $v$]\label{Lemma:decrease:F:v}
Let $\{u^t,v^t,s^t\}_{t}$ be the sequence generated from Algorithm~\ref{Alg:1}.
The following inequality holds for any $k\ge1$:
\[
F(u^{t+1},v^{t+1},s^t) - F(u^{t+1},v^{t},s^t)
\le 
 -\frac{1}{2}
\|1/m - \pi(u^{t+1}, v^t, s^t)\trans 1\|_1^2.
\]
where
\[
\|1/m - \pi(u^{t+1}, v^t, s^t)\|_1 = \sum_j\left|\frac{1}{m}- \sum_i{\pi_{i,j}}(u^{t+1}, v^t, s^t)\right|.
\]
\end{lemma}
\begin{proof}[Proof of Lemma~\ref{Lemma:decrease:F:v}]
According to the expression of $F$, we have that
\begin{align*}
&F(u^{t+1},v^{t+1},s^t) - F(u^{t+1},v^{t},s^t)\\
=&\sum_{i,j}{\pi_{i,j}}(u^{t+1},v^{t+1},s^t) - 
\sum_{i,j}{\pi_{i,j}}(u^{t+1},v^{t},s^t)
+\frac{1}{m}\sum_{j=1}^m(v_j^{t} - v_j^{t+1})\\
=&\frac{1}{m}\sum_{j=1}^m(v_j^{t} - v_j^{t+1})
=-\frac{1}{m}\sum_{j=1}^m\log\frac{1/m}{\sum_i{\pi_{i,j}}(u^{t+1}, v^t, s^t)},
\end{align*}
where the second equality is because that 
\begin{align*}
\sum_i{\pi_{i,j}}(u^{t+1}, v^{t+1}, s^t)&=\frac{1}{m},\\
\sum_j{\pi_{i,j}}(u^{t+1}, v^t, s^t) &= \frac{1}{n}.
\end{align*}
Therefore, applying the Pinsker's inequality in Theorem~\ref{Theorem:Pinsker} implies that
\[
F(u^{t+1},v^{t+1},s^{t}) - F(u^{t+1},v^{t},s^{t})
\le -\frac{1}{2}
\left(
\sum_j\left|\frac{1}{m}- \sum_i{\pi_{i,j}}(u^{t+1}, v^t, s^t)\right|
\right)^2.
\]
\end{proof}
\begin{lemma}[Decrease of $F$ in $u$]\label{Lemma:decrease:F:u}
Let $\{u^t,v^t,s^t\}_{t}$ be the sequence generated from Algorithm~\ref{Alg:1}.
The following inequality holds for any $t\ge1$:
\[
F(u^{t+1},v^{t},s^{t}) - F(u^{t},v^{t},s^{t})
\le 
 -\frac{1}{2}
\|1/n - \pi(u^t,v^t,s^t)1\|_2^2.
\]
where
\[
\|1/n - \pi(u^t,v^t,s^t)1\|_2^2 = \sum_i\left|\frac{1}{n}- \sum_j{\pi_{i,j}}(u^{k}, v^k, t^k)\right|^2.
\]
\end{lemma}
\begin{proof}[Proof of Lemma~\ref{Lemma:decrease:F:u}]
For fixed $i\in[n]$, define
\[
h_i = \sum_j\pi_{i,j}(u^{t+1},v^{t},s^{t}) - \sum_j\pi_{i,j}(u^{t},v^{t},s^{t}) - \frac{1}{n}\log\frac{1/n}{\sum_j{\pi_{i,j}}(u^t,v^t,s^t)}
\]
According to the expression of $F$, 
\begin{align*}
F(u^{t+1},v^{t},s^{t}) - F(u^{t},v^{t},s^{t})
=\sum_ih_{i},
\end{align*}
and it suffices to provide an upper bound for $h_i, i\in[n]$.
By substituting the expression of $u^{t+1}$ into $h_i$, we have that
\begin{align*}
h_i&=\sum_j\pi_{i,j}(u^{t},v^{t},s^{t})
\left[
\frac{1/n}{\sum_j{\pi_{i,j}}(u^t,v^t,s^t)} - 1
\right]-\frac{1}{n}\log\frac{1/n}{\sum_j{\pi_{i,j}}(u^t,v^t,s^t)}\\
&=\frac{1}{n} - \big(\pi(u^t,v^t,s^t)1\big)_i - \frac{1}{n}\log\frac{1/n}{\big(\pi(u^t,v^t,s^t)1\big)_i}
\end{align*}
Define the function
\[
\ell(x) = \frac{1}{n} - x - \frac{1}{n}\log\frac{1/n}{x} + (x-1/n)^2.
\]
We can see that this function attains its maximum at $x=1/n$, with $\ell(1/n)=0$.
It follows that
\[
h_i\le -\left(
\big(\pi(u^t,v^t,s^t)1\big)_i - \frac{1}{n}
\right)^2.
\]
The proof is completed.
\end{proof}

\begin{lemma}\label{Lemma:stationary}
Let $\{u^t,v^t,s^t\}_{t}$ be the sequence generated from Algorithm~\ref{Alg:1},
which is terminated when the following conditions hold:
\[
\|\xi^{t+1}\|\le\epsilon_1,\quad
\|1/n - \pi(u^t,v^t,s^t)1\|_2\le\frac{\epsilon_2}{4\|AU\|_\infty^2},\quad
\|1/m - \pi(u^{t+1}, v^t, s^t)\trans1\|_1\le\frac{\epsilon_2}{4\|AU\|_\infty^2}.
\]
Then $\{u^T,v^T,s^T\}$ is an $(\epsilon_1,\epsilon_2)$ stationary point of \eqref{Eq:final:min:three:block}.
\end{lemma}
\begin{proof}[Proof of Lemma~\ref{Lemma:stationary}]
The condition $\|\xi^{t+1}\|\le\epsilon_1$ directly implies that
\[
\|\text{Grad}_sF(u^T,v^T,s^T)\|\le\epsilon_1.
\]
Suppose that
\[
\pi(u^T,v^T,s^T)1=r,\quad \pi(u^{T},v^T,s^T)\trans1=c,
\]
where $\|1/n-r\|_2\le\epsilon_2/(4\|AU\|_\infty^2)$ and $\|1/m-c\|_1\le\epsilon_2/(4\|AU\|_\infty^2)$.
Then we find that 
\[
F(u^T,v^T,s^T) = \min_{\pi}~\left\{\sum_{i,j}\pi_{i,j}M_{i,j} - \eta H(\pi):~
\sum_{j}\pi_{i,j}=r_i,
\sum_i\pi_{i,j}=c_j
\right\},
\]
and
\[
\min_{u,v}F(u,v,s^T) = \min_{\pi}~\left\{\sum_{i,j}\pi_{i,j}M_{i,j} - \eta H(\pi):~
\sum_{j}\pi_{i,j}=\frac{1}{n},
\sum_i\pi_{i,j}=\frac{1}{m}
\right\},
\]
where $M_{i,j} = \|A_{i,j}Us^T\|_2^2$.
It follows that 
\begin{align*}
&F(u^T,v^T,s^T)-\min_{u,v}F(u,v,s^T)\\
\le &\eta\log(mn) + 2\|1/m-c_1\|_1\times\|AU\|_\infty^2\le \epsilon_2,
\end{align*}
where the last inequality is by taking $\eta=\epsilon_2/(2\log(mn))$.

\end{proof}

\begin{lemma}[Lower Boundedness of $F$]\label{Lemma:boundedness:F}
Denote by $(u^*,v^*,s^*)$ the global optimum of \eqref{Eq:final:min:three:block}.
Then we have that
\[
F(u^*,v^*,s^*)\ge 1 - \frac{1}{\eta}\|AU\|_\infty^2.
\]
\end{lemma}
\begin{proof}[Proof of Lemma~\ref{Lemma:boundedness:F}]
It is easy to show that
\[
\sum_{i,j}\pi_{i,j}(u^*,v^*,s^*)=1.
\]
Moreover, for any $(i,j)$, we have that $c_{i,j}\le \|AU\|_\infty^2$.
It follows that
\[
\exp\left(
-\frac{1}{\eta}\|AU\|_\infty^2
+
u_i^*+v_j^*
\right)
\le
\pi_{i,j}\le 1,
\]
and therefore $u_i^*+v_j^*\le \frac{1}{\eta}\|AU\|_\infty^2$ for any $(i,j).$
Hence we conclude that
\[
\sum_{i,j}\pi_{i,j}(u^*,v^*,s^*)
-\frac{1}{n}\sum_{i=1}^nu_i-\frac{1}{m}\sum_{j=1}^mv_j
\ge 1 - \frac{1}{\eta}\|AU\|_\infty^2.
\]
\end{proof}

In the following we give a re-statement of Theorem~\ref{Theorem:convergence:BCD} and the formal proof.
\begin{theorem*}[Re-statement of Theorem~\ref{Theorem:convergence:BCD}]
Choose parameters 
\[
\tau=\frac{1}{4\|AU\|_\infty^2L_2/\eta+\varrho L_1^2},\quad
\eta=\frac{\epsilon_2}{2\log(mn)},\quad
\varrho=\frac{2\|AU\|_\infty^2}{\eta} + \frac{4\|AU\|_\infty^4}{\eta^2},
\]
and Algorithm~\ref{Alg:1} terminates when
\[
\|\xi^{t+1}\|\le\epsilon_1,\quad
\|1/n - \pi(u^t,v^t,s^t)1\|_2\le\frac{\epsilon_2}{4\|AU\|_\infty^2},\quad
\|1/m - \pi(u^{t+1}, v^t, s^t)\trans1\|_1\le\frac{\epsilon_2}{4\|AU\|_\infty^2}.
\]
We say that $(\hat{u}, \hat{v}, \hat{s})$ is a $(\epsilon_1,\epsilon_2)$-stationary point
of \eqref{Eq:final:min:three:block} if 
\begin{align*}
    \|\text{Grad}_{s}F(\hat{u}, \hat{v}, \hat{s})\|&\le \epsilon_1,\\
    F(\hat{u}, \hat{v}, \hat{s}) - \min_{u,v}F(u,v,\hat{s})&\le \epsilon_2,
\end{align*}
where $\text{Grad}_{s}F(u,v,s)$ denotes the partial derivative of $F$ with respect to the variable $s$ on the sphere $\mathbb{S}^{d(n+m)-1}$.
Then Algorithm~\ref{Alg:1} returns an $(\epsilon_1,\epsilon_2)$-stationary point in iterations
\[
T = \mathcal{O}\left(
\log(mn)\cdot
\left[ 
\frac{1}{\epsilon_2^3} + \frac{1}{\epsilon_1^2\epsilon_2}
\right]
\right).
\]
\end{theorem*}

\begin{proof}[Proof of Theorem~\ref{Theorem:convergence:BCD}]
We can build the one-iteration descent result based on Lemma~\ref{Lemma:decrease:F:s}, Lemma~\ref{Lemma:decrease:F:v}, and Lemma~\ref{Lemma:decrease:F:u}:
\begin{align*}
    & F(u^{t+1},v^{t+1},s^{t+1}) - F(u^t,v^t,s^t)\\
\le & -\left(
\frac{1}{2}\|1/n-\pi(u^t,v^t,s^t)1\|_2^2 + \frac{1}{2}\|1/m-\pi(u^{t+1},v^t,s^t)\trans1\|_1^2 + 
\frac{1}{8\|AU\|_\infty^2L_2/\eta+2\varrho L_1^2}\|\xi^{t+1}\|^2_2
\right)\\
=&-\frac{1}{2}\left(
\|1/n-\pi(u^t,v^t,s^t)1\|_2^2
+
\|1/m-\pi(u^{t+1},v^t,s^t)\trans1\|_1^2
+
\frac{\eta^2\|\zeta^{t+1}\|^2}{2\|AU\|_\infty^2\eta(2L_2 + L_1^2) + 4\|AU\|_\infty^4L_1^2}
\right)
\end{align*}
Then we have that
\begin{align*}
&F(u^T, v^T, s^T) - F(u^0, v^0, s^0)\\
\le &-\frac{1}{2}\sum_{t=0}^{T-1}\left(
\|1/n-\pi(u^t,v^t,s^t)1\|_2^2
+
\|1/m-\pi(u^{t+1},v^t,s^t)\trans1\|_1^2
+
\frac{\eta^2\|\zeta^{t+1}\|^2}{2\|AU\|_\infty^2\eta(2L_2 + L_1^2) + 4\|AU\|_\infty^4L_1^2}
\right)\\
\le&-\frac{1}{2}\cdot
\min\left\{
1,\frac{1}{2\|AU\|_\infty^2\eta(2L_2 + L_1^2) + 4\|AU\|_\infty^4L_1^2}
\right\}\\&\qquad\qquad\times
\sum_{t=0}^{T-1}\left(
\|1/n-\pi(u^t,v^t,s^t)1\|_2^2
+
\|1/m-\pi(u^{t+1},v^t,s^t)\trans1\|_1^2
+
\eta^2\|\xi^{t+1}\|^2_2
\right)\\
\le&
-\frac{1}{2}T\cdot\min\left\{
1,\frac{1}{2\|AU\|_\infty^2\eta(2L_2 + L_1^2) + 4\|AU\|_\infty^4L_1^2}
\right\}\cdot\min\left\{\epsilon_1^2, \frac{\epsilon_2^2}{16\|AU\|_\infty^4}, \frac{\epsilon_2^2}{16\|AU\|_\infty^4}\right\}.
\end{align*}
Therefore,
\begin{align*}
T\le& [F(u^0,v^0,t^0) - F(u^T, v^T, s^T)]\max\left\{
2, 
4\|AU\|_\infty^2\eta(2L_2 + L_1^2) + 8\|AU\|_\infty^4L_1^2
\right\}\\
&\qquad\qquad\max\left\{\frac{1}{\epsilon_1^2}, \frac{16\|AU\|_\infty^4}{\epsilon_2^2}, \frac{16\|AU\|_\infty^4}{\epsilon_2^2}\right\}\\
\le&
\left(
F(u^0,v^0,t^0) - 1 + \frac{\|AU\|_\infty^2}{\eta}
\right)\max\left\{
2, 
4\|AU\|_\infty^2\eta(2L_2 + L_1^2) + 8\|AU\|_\infty^4L_1^2
\right\}\\
&\qquad\qquad\max\left\{\frac{1}{\epsilon_1^2}, \frac{16\|AU\|_\infty^4}{\epsilon_2^2}, \frac{16\|AU\|_\infty^4}{\epsilon_2^2}\right\}\\
=&\mathcal{O}\left(
\log(mn)\cdot\left[
\frac{1}{\epsilon_2^3}+\frac{1}{\epsilon_1^2\epsilon_2}
\right]
\right).
\end{align*}
\end{proof}

\clearpage
\section{
TECHNICAL PROOFS IN SECTION~\ref{Sec:inference}
}
\label{app:inference:theorem}
\subsection{Proof of Theorem~\ref{Proposition:finite:KPW:preliminary}}

\begin{proof}[Proof of Lemma~\ref{Lemma:10}]
Denote $\mathcal{F} = \{f\in\mathcal{H}:~\|f\|_{\mathcal{H}}\le 1\}$.
By the bias-variation decomposition, we have that
\begin{align*}
\mathbb{E}[\powerP{\KPW(\hmu_n, \mu)}]&\le 
\sup_{f\in\mathcal{F}}~\mathbb{E}[\powerP{W(f\#\hmu_n, f\#\mu)}] 
\\&\qquad\qquad+ 
\mathbb{E}\left[ 
\sup_{f\in\mathcal{F}}\left(
\powerP{W(f\#\hmu_n, f\#\mu)} - \mathbb{E}[\powerP{W(f\#\hmu_n, f\#\mu)}]
\right)
\right].
\end{align*}
For fixed $f\in\mathcal{F}$, we can see that 
\begin{align*}
\mathbb{E}[\powerP{W(f\#\hmu_n, f\#\mu)}]
&\le
c_pn^{-\frac{1}{(2p)\lor d}}(\log n)^{\zeta_{p,d}/p}
\end{align*}
where $c_p$ is a constant depending only on $p$ and 
\[
\zeta_{p,d} = \left\{
\begin{aligned}
1,&\quad\text{if }d=2p,\\
0,&\quad\text{otherwise}.
\end{aligned}
\right.
\]

Now we start to upper bound the variation term.
Define the empirical process 
\[
X_{f} = \powerP{W(f\#\hmu_n, f\#\mu)} - \mathbb{E}[\powerP{W(f\#\hmu_n, f\#\mu)}].
\]
It is easy to see that $\mathbb{E}[X_{f}]=0$.
Moreover, we can show that for fixed $f$, the random variable $X_f$ is sub-exponential.
Denote by $Z=\{z_i\}_{i=1}^n$ and $Z'=\{z_i'\}_{i=1}^n$ i.i.d. samples from $f\#\mu$.
Take $g(Z)=\powerP{W(f\#\hmu_n, f\#\mu)}$.
Then we have that 
\begin{align*}
    |g(Z) - g(Z_{(i)}')|&\le \powerP{W(f\#\hmu_n, f\#\hmu_n')}\le n^{-1/(2\lor p)}\|Z - Z'\|_2.
\end{align*}
It follows that 
\[
\sum_{i=1}^n\|\nabla_ig(Z)\|^2\le n^{-2/(2\lor p)},\quad 
\max_{1\le i\le n}\|\nabla_ig(Z)\|\le n^{-1/p}.
\]
Then the Poincare's inequality in Theorem~\ref{Theorem:poincare} implies that 
\[
\text{Pr}
\{
X_f\ge t
\}
\le \exp\left(
-K^{-1}\min\{tn^{1/p}, t^2n^{2/(2\lor p)}\}
\right).
\]
Hence we conclude that $X_f$ is sub-exponential with parameters $(\sqrt{K/2}n^{-1/(2\lor p)}, (K/2)n^{-1/p})$.

For the function space $\mathcal{F}$, define the corresponding metric
\[
\textsf{d}(f,f') = \|f - f'\|_{\mathcal{H}}.
\]
Let $X\sim \mu$.
Then for any $f,f'\in\mathcal{F}$, we have that
\begin{align*}
    &|X_{f} - X_{f'}|\\
\le&\mathbb{E}\big[ 
\powerP{W(f\#\hmu_n, f'\#\hmu_n)} + \powerP{W(f\#\mu, f'\#\mu)}
\big] + \mathbb{E}\big[ 
\powerP{W(f\#\hmu_n, f'\#\hmu_n)} + \powerP{W(f\#\mu, f'\#\mu)}
\big]\\
\le&2\powerP{\mathbb{E}\|f(X) - f'(X)\|_2^p} + \left(
\frac{1}{n}\sum_{i=1}^n\|f(X_i) - f'(X_i)\|_2^p
\right)^{1/p} + \mathbb{E}\left[ 
\left(
\frac{1}{n}\sum_{i=1}^n\|f(X_i) - f'(X_i)\|_2^p
\right)^{1/p}
\right].
\end{align*}

Note that the following upper bound holds for any $f,f'\in\mathcal{F}$ and $x\in\mathbb{R}^D$:
\begin{align*}
\|f(x) - f'(x)\|_2&=\max_{a:~\|a\|_2\le 1}~\inp{f(x) - f'(x)}{a}\\
&=\max_{a:~\|a\|_2\le 1}~\inp{f(x)}{a} - \inp{f'(x)}{a}\\
&=\max_{a:~\|a\|_2\le 1}~\inp{f}{K_xa}_{\mathcal{H}_K} - \inp{f'}{K_xa}_{\mathcal{H}_K}\\
&=\max_{a:~\|a\|_2\le 1}~\inp{f - f'}{K_xa}_{\mathcal{H}_K}\\
&\le \|f-f'\|_{\mathcal{H}_K}\times \max_{a:~\|a\|_2\le 1}~\|K_xa\|_{\mathcal{H}_K}\\
&=\|f-f'\|_{\mathcal{H}_K}\times \max_{a:~\|a\|_2\le 1}~\sqrt{a\trans K(x,x)a}\\
&=\sqrt{B}\|f-f'\|_{\mathcal{H}_K}.
\end{align*}
As a consequence, substituting this upper bound into the relation above implies that 
\[
|X_{f} - X_{f'}| \le 4\sqrt{B}\textsf{d}(f,f').
\]
Applying the $\epsilon$-net argument similar to the Dudley's entropy integral bound~\citep[Theorem~5.22]{wainwright2019high} gives
\[
\mathbb{E}\bigg[
\sup_{f\in\mathcal{F}}
X_{f}
\bigg]
\le \inf_{\epsilon>0}
\left\{
4\sqrt{B}\epsilon + \sqrt{2K}n^{-1/(2\lor p)}\sqrt{\log\mathcal{N}(\mathcal{F}, \textsf{d}, \epsilon)} + (K/2)n^{-1/p}\log\mathcal{N}(\mathcal{F}, \textsf{d}, \epsilon)
\right\}
\]
Taking $\mathcal{N}(\mathcal{F}, \textsf{d}, \epsilon)=\left\lceil\frac{1}{\epsilon}\right\rceil$ and $\epsilon=n^{-1/p}$ implies that
\[
\mathbb{E}\bigg[
\sup_{f\in\mathcal{F}}
X_{f}
\bigg]\lesssim
n^{-1/(2\lor p)}\sqrt{\log(n)} + n^{-1/p}\log(n).
\]

\end{proof}

\begin{proof}[Proof of Lemma~\ref{Lemma:variance}]
We start to upper bound the variance term
\[
\powerP{\KPW(\hmu_n, \mu)} - \mathbb{E}[ \powerP{\KPW(\hmu_n, \mu)}].
\]
Denote by $X=\{x_i\}_{i=1}^n$ and $X'=\{x_i'\}_{i=1}^n$ i.i.d. samples from $\mu$, and let $g(X) = \powerP{\KPW(\hmu_n, \mu)}$.
Based on the triangular inequality, we find that 
\begin{align*}
    |g(X) - g(X')|&\le n^{-1/p}\left(
\sum_{i=1}^n\max_{f\in\mathcal{F}}~\|f(x_i) - f(x_i')\|_2
\right)^{1/p}\\
&\le n^{-1/p}\left(
\sum_{i=1}^nL\|x_i-x_i'\|%
\right)^{1/p}\\
&\le n^{-1/(2\lor p)}L^{1/p}\|X - X'\|.
\end{align*}
It follows that 
\[
\sum_{i=1}^n\|\nabla_ig(Z)\|^2\le n^{-2/(2\lor p)}L^{2/p},\quad 
\max_{1\le i\le n}\|\nabla_ig(Z)\|\le n^{-1/p}L^{1/p}.
\]
Then the Poincare's inequality in Theorem~\ref{Theorem:poincare} implies that 
\[
\text{Pr}
\{
\left|\powerP{\KPW(\hmu_n, \mu)} - \mathbb{E}[ \powerP{\KPW(\hmu_n, \mu)}]\right|\ge t
\}
\le \exp\left(
-K^{-1}\min\{tn^{1/p}L^{-1/p}, t^2n^{2/(2\lor p)}L^{-2/p}\}
\right).
\]
Substituting the right-hand-side with $\alpha$ completes the proof.
\end{proof}

\begin{proof}[Proof of Theorem~\ref{Proposition:finite:KPW:preliminary}]
Based on the triangular inequality, we can see that 
\begin{align*}
&\big| \powerP{\KPW(\hmu_n, \hnu_m)} - \powerP{\KPW(\mu,\nu)}\big|
\le\powerP{\KPW(\hmu_n, \mu)} + \powerP{\KPW(\hnu_m, \nu)}.
\end{align*}
It suffices to upper bound $\powerP{\KPW(\hmu_n, \mu)}$ and $\powerP{\KPW(\hnu_m, \nu)}$ separately.
By the bias-variance decomposition, 
\begin{align*}
    \powerP{\KPW(\hmu_n, \mu)}&\le \mathbb{E}[ \powerP{\KPW(\hmu_n, \mu)}] + \bigg(
    \powerP{\KPW(\hmu_n, \mu)} - \mathbb{E}[ \powerP{\KPW(\hmu_n, \mu)}]
    \bigg),
\end{align*}
where the first term quantifies the bias for empirical estimation, and the second term quantifies the variance of estimation.
The bias term can be upper bounded by applying Lemma~\ref{Lemma:10}, and the variance term can be upper bounded by applying Lemma~\ref{Lemma:variance}.
In summary, with probability at least $1-\alpha$, it holds that 
\begin{align*}
\powerP{\KPW(\hmu_n, \mu)}&\lesssim \max\left\{
n^{-1/p}K\log(1/\alpha), n^{-1/(2\lor p)}\sqrt{K\log(1/\alpha)}
\right\}L^{1/p}\\
&\quad + n^{-\frac{1}{(2p)\lor d}}(\log n)^{\zeta_{p,d}/p}
+
n^{-1/(2\lor p)}\sqrt{\log(n)} + n^{-1/p}\log(n).
\end{align*}
The upper bound for $\powerP{\KPW(\hnu_m, \nu)}$ can be proceeded similarly.

\end{proof}

\subsection{Testing Performance}

Based on the finite-sample guarantee in Theorem~\ref{Proposition:finite:KPW:preliminary}, we are able to characterize the performance of the KPW test. 
To make the type-I error below than $\alpha$, we reject the null hypothesis as long as the empirical statistic $\KPW(\hmu_n,\hnu_m)\ge\gamma_{m,n}$, where
\begin{align*}
    \gamma_{m,n}^{1/p} &\sim\max\left\{
N^{-1/p}K\log(1/\alpha), N^{-1/(2\lor p)}\sqrt{K\log(1/\alpha)}
\right\}L^{1/p}\\
&\qquad + N^{-\frac{1}{(2p)\lor d}}(\log N)^{\zeta_{p,d}/p}
+
N^{-1/(2\lor p)}\sqrt{\log(n)} + N^{-1/p}\log(n).
\end{align*}
For the alternative hypothesis, assume that target distributions $\mu$ and $\nu$ satisfy $\KPW(\mu,\nu)>\gamma_{m,n}$.
Then the type-II error can be upper bounded as 
\begin{align*}
    &\Pr{}_{\mathcal{H}_1}\bigg(
    \KPW(\hmu_n,\hnu_m)<\gamma_{m,n}
    \bigg)\\
    =&\Pr{}_{\mathcal{H}_1}\bigg(
    \KPW(\hmu_n,\hnu_m) - \KPW(\mu,\nu)<\gamma_{m,n} - \KPW(\mu,\nu)
    \bigg)\\
    =&\Pr{}_{\mathcal{H}_1}\bigg(
    \KPW(\mu,\nu) - \KPW(\hmu_n,\hnu_m)> \KPW(\mu,\nu) - \gamma_{m,n}
    \bigg)\\
    \le&\Pr{}_{\mathcal{H}_1}\bigg(
    |\KPW(\mu,\nu) - \KPW(\hmu_n,\hnu_m)| > \KPW(\mu,\nu) - \gamma_{m,n}
    \bigg)\\
    \le&\frac{\mathbb{E}\left(\KPW(\mu,\nu) - \KPW(\hmu_n,\hnu_m)\right)^2}{\big(\KPW(\mu,\nu) - \gamma_{m,n}\big)^2}.
\end{align*}

\subsection{Finite-sample Guarantee for \texorpdfstring{$p\in[1,2)$}{p}}
\label{Sec:sub:finite:sample:1:2}
In this subsection, we discuss the finite-sample guarantee for KPW distance with $p$-Wasserstein distance for $p\in[1,2)$.
Note that it is not necessary to rely on the Poincare inequality or projection poincare inequality to obtain the result.
We first present several technical lemmas before showing the final result.
\begin{lemma}\label{Lemma:norm:f}
Based on Assumption~\ref{Assumption:1}, for $f\in\{f\in\mathcal{H}:~\|f\|_{\mathcal{H}}\le 1\}$, we have
\[
\|f(x)\|_2\le \sqrt{B},\quad \forall x\in\mathbb{R}^D.
\]
\end{lemma}
\begin{proof}[Proof of Lemma~\ref{Lemma:norm:f}]
For fixed $x\in\mathcal{X}$, the norm of $f(x)$ can be upper bounded as the following:
\[
\|f(x)\|_2^2=\inp{f(x)}{f(x)}=\inp{f}{K_xf(x)}_{\mathcal{H}}\le \|f\|_{\mathcal{H}}\|K_xf(x)\|_{\mathcal{H}}
\le \|K_xf(x)\|_{\mathcal{H}}.
\]
In particular,
\begin{align*}
    \|K_xf(x)\|_{\mathcal{H}}^2&=
\inp{K_xf(x)}{K_xf(x)}_{\mathcal{H}}
\\&=
\inp{\big(
K_xf(x)
\big)f(x)}{f(x)}
\\&=
\inp{K(x,f(x))f(x)}{f(x)}
\\&=f(x)\trans K(x,f(x))f(x)
\\&\le B\|f(x)\|^2_2
\end{align*}
Combining those two relations above implies the desired result.
\end{proof}

\begin{lemma}\label{Lemma:bias:p12}
For $p\in[1,2)$, the bias term of empirical KPW distance can be upper bounded as 
\[
\mathbb{E}[\powerP{\KPW(\hmu_n, \mu)}]\lesssim
n^{-\frac{1}{(2p)\lor d}}(\log n)^{\zeta_{p,d}/p}
+
n^{1/2-1/p}\sqrt{\log(n)}+ n^{-1/p}.
\]
where $\zeta_{p,d}=1$ if $d=2p$ and $\zeta_{p,d}=0$ otherwise.
\end{lemma}
\begin{proof}[Proof of Lemma~\ref{Lemma:bias:p12}]
Following the similar argument as in Lemma~\ref{Lemma:10}, we can see that 
\begin{align*}
\mathbb{E}[\powerP{\KPW(\hmu_n, \mu)}]&\le 
\sup_{f\in\mathcal{F}}~\mathbb{E}[\powerP{W(f\#\hmu_n, f\#\mu)}] 
\\&\qquad+ 
\mathbb{E}\left[ 
\sup_{f\in\mathcal{F}}\left(
\powerP{W(f\#\hmu_n, f\#\mu)} - \mathbb{E}[\powerP{W(f\#\hmu_n, f\#\mu)}]
\right)
\right],
\end{align*}
and the first term can also be bounded similarly.
To upper bound the second term, define the empirical process $\{X_f\}$ as in Lemma~\ref{Lemma:10}.
For fixed $f$, the random variable $X_f$ can be shown to be sub-Gaussian.
Denote by $Z=\{z_i\}_{i=1}^n$ and $Z_{(i)}'$ the sample set so that the $i$-th element is different. 
Take $g(Z)=\powerP{W(f\#\hmu_n, f\#\mu)}$.
Then we have that 
\begin{align*}
    |g(Z) - g(Z_{(i)}')|&\le \powerP{W(f\#\hmu_n, f\#\hmu_n')}\le \left(
    \frac{1}{n}\|f(z_i) - f(z_i')\|^p_2
    \right)^{1/p}\\
    &\le 
    n^{-1/p}2\sqrt{B}.
\end{align*}
Therefore, applying the McDiarmid’s inequality in Theorem~\ref{The:Mcdiarmid} implies
\[
\text{Pr}
\{
\left|X_f\right|\ge u
\}
\le 2\exp\left(
-\frac{u^2}{2Bn^{1-2/p}}
\right).
\]
Applying Lemma~\ref{Lemma:equivalent:subGaussian} implies that for fixed $\ell$, the random variable $X_f$ is sub-Gaussian with the parameter $\sigma^2=36Bn^{1-2/p}$.
Then applying the $\epsilon$-net argument similar to the Dudley's entropy integral bound~\citep[Theorem~5.22]{wainwright2019high} gives
\[
\mathbb{E}\bigg[
\sup_{f\in\mathcal{F}}
X_{f}
\bigg]
\le \inf_{\epsilon>0}
\left\{
4\sqrt{B}\epsilon + \sqrt{36Bn^{1-2/p}}\sqrt{2\log\mathcal{N}(\mathcal{F}, \textsf{d}, \epsilon)}
\right\}.
\]
Taking $\mathcal{N}(\mathcal{F}, \textsf{d}, \epsilon)=\left\lceil\frac{1}{\epsilon}\right\rceil$ and $\epsilon=n^{-1/p}$ implies that
\[
\mathbb{E}\bigg[
\sup_{f\in\mathcal{F}}
X_{f}
\bigg]\lesssim
n^{1/2-1/p}\sqrt{\log(n)}+ n^{-1/p}.
\]
\end{proof}

\begin{lemma}\label{Lemma:variance:p12}
For $p\in[1,2)$, with  with probability at least $1-\alpha$, it holds that 
\[
\left|\powerP{\KPW(\hmu_n, \mu)} - \mathbb{E}[ \powerP{\KPW(\hmu_n, \mu)}]\right|
\le
n^{1/2-1/p}\sqrt{2B\log\frac{2}{\alpha}}.
\]
\end{lemma}
\begin{proof}[Proof of Lemma~\ref{Lemma:variance:p12}]
Denote by $Z=\{z_i\}_{i=1}^n$ and $Z_{(i)}'$ the sample set so that the $i$-th element is different. 
Take $g(Z)=\powerP{\KPW(\hmu_n, \mu)}$.
Then we can see that 
\[
|g(Z) - g(Z_{(i)}')|\le \powerP{\KPW(\hmu_n, \hmu_n')}\le n^{-1/p}2\sqrt{B}.
\]
Then applying the McDiarmid’s inequality in Theorem~\ref{The:Mcdiarmid} implies
\[
\text{Pr}
\left\{
\left|\powerP{\KPW(\hmu_n, \mu)} - \mathbb{E}[ \powerP{\KPW(\hmu_n, \mu)}]\right|\ge u
\right\}
\le 2\exp\left(
-\frac{u^2}{2Bn^{1-2/p}}
\right).
\]
\end{proof}
Based on Lemma~\ref{Lemma:bias:p12} and Lemma~\ref{Lemma:variance:p12}, we obtain the uncertainty quantification result in Theorem~\ref{theorem:p12:UQ}.

\clearpage
\section{
IMPLEMENTATION DETAILS FOR COMPUTING KPW DISTANCE
}
\label{Appendix:Implementation}

The variable $s$ is initialized to be a uniform random vector over sphere.
The dual variable $v$ is initialized to be a Gaussian random vector with unit covariance.
When updating the block of variables $u^{t+1}$ and $v^{t+1}$, we make the change of variables $(u')^{t+1} = \exp(u^{t+1})$ and $(v')^{t+1} = \exp(v^{t+1})$. We update $(u')^{t+1}$ and $(v')^{t+1}$ instead to accelerate the computation:
\begin{align*}
    (u')^{t+1} &= \left\{\frac{1/n}{\sum_j\exp\left(-\frac{1}{\eta}c_{i,j}+(v_j')^t\right)}\right\}_i\\
    (v')^{t+1} &= \left\{\frac{1/m}{\sum_i\exp\left(-\frac{1}{\eta}c_{i,j}+(u_i')^{t+1}\right)}\right\}_j,
\end{align*}
and we further store the matrix $A$ with $A_{i,j}=\exp\left(-\frac{1}{\eta}c_{i,j}\right)$ in advance to reduce the computational cost.
The transport mapping $\pi^{t+1}\triangleq(\pi_{i,j}(u^{t+1}, v^{t+1}, s^t))_{i,j}$ can be formulated without going through a for loop but only with multiplication operators:
\[
\pi^{t+1} = (u')^{t+1}~\texttt{.*}~A~\texttt{.*}~[(v')^{t+1}]\trans,
\]
where the operator \texttt{.*} means we multiply two objects componentiwisely in terms of array broadcasting.
When updating $\zeta^{t+1}$, we first formulate the matrix $V^{t+1}$ with
\[
V^{t+1}_{i,j} = \sum_{i,j}\pi^{t+1}_{i,j}A_{i,j}\trans A_{i,j}
\]
and then continue the matrix multiplication procedure in \eqref{Eq:update:zeta:closed}.
Denote by $G_i$ the $i$-th row block of the gram matrix $G$, then 
\begin{align*}
V^{t+1} &= \left\{\sum_{i,j}\pi^{t+1}_{i,j}(G_i + G_{n+j})\trans(G_i + G_{n+j})\right\}_{i,j}\\
&=\left\{\sum_{i,j}\pi^{t+1}_{i,j}(G_i\trans G_i + G_{n+j}\trans G_{n+j} + G_{n+j}\trans G_i + G_i\trans G_{n+j})\right\}_{i,j}.
\end{align*}
Consequently, we can compute each of the four components in the formula above without executing double for loops and then sum them up to obtain the matrix $V^{t+1}$.
During the numerical implementation, we also find that the computation is sensitive to the choice of $\eta$. 
 This phenomenon has also been observed when using Sinkhorn's algorithm to compute Wasserstein distance or projected Wasserstein distance. 
 When $\eta$ is too small, the iteration update may have numerical instability issues. 
 When $\eta$ is too large, the obtained solution is far away from the optimal solution to the original KPW distance. 
 We have tried the best to tune this parameter to make the algorithm maintain the best performance. 
 How to tune this hyper-parameter systematically is left for future works.

\clearpage
\section{
DETAILS ABOUT EXPERIMENT
}\label{Appendix:Experiment:detail}

\subsection{Sample Complexity}
\label{Appendix:experiment:sample:detail}

\begin{figure*}[t]
\centering
\includegraphics[width=\textwidth]{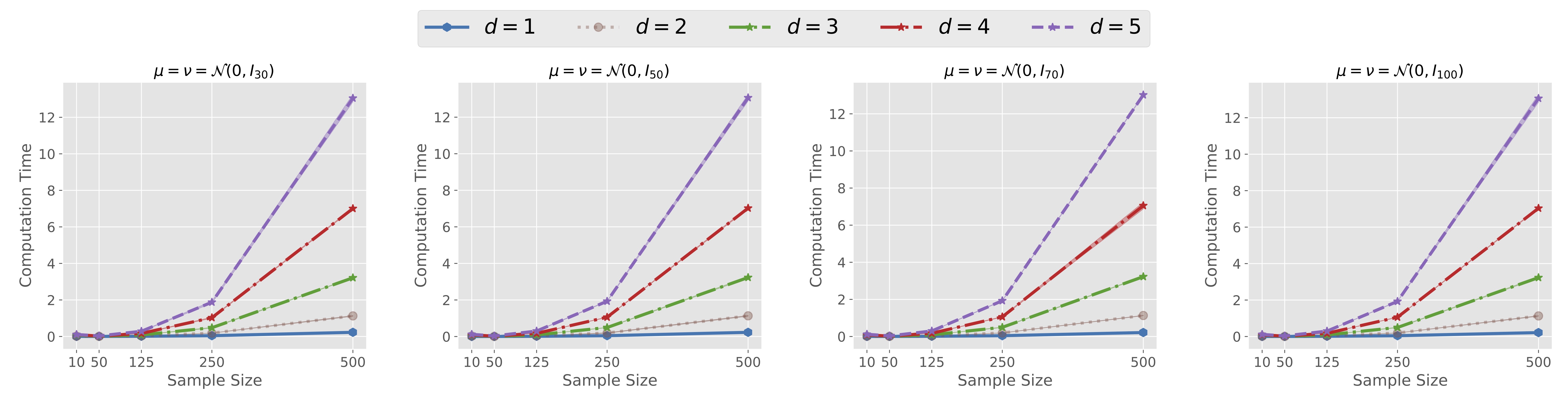}
\caption{Mean computation time for computing $\KPW(\hmu_n, \hnu_n)$ for varying $n$. Results are averaged over $10$ independent trials.}
\label{Fig:sample:size:time}
\end{figure*}

In this experiment, we fix hyper-parameters $\sigma^2=\texttt{1}, \rho=\texttt{0.5}$ for computing KPW distances.
The values of empirical KPW distances across different choices of sample size are reported in Figure~\ref{fig:sample:Gaussian}, and the corresponding computation time is reported in Figure~\ref{Fig:sample:size:time}.
From the plot we can see that it is efficient to compute KPW distances with reasonably small sample size $n$ and projected dimension $d$.

\subsection{Configurations}\label{Appendix:config:baseline}
All methods are implemented using python 3.7 (Pytorch 1.1) on a MacBook Pro labtop with 32GB of memory.
When running the code, there is no swapping of memory and the average CPU frequency is $3.2$~GHZ.
We compute the projected Wasserstein distance based on the official code in \url{https://github.com/fanchenyou/PRW}.
We run the MMD-O test based on the code in \url{https://github.com/fengliu90/DK-for-TST}.
We run the MMD-NTK test based on the code in \url{https://github.com/xycheng/NTK-MMD}.
From extensive experiments we realize that MMD-NTK is the most computationally efficient test, but its power does not scale the best.
On the other hand, this method can be useful when performing a test for the large-sampled case, while our method may be intractable to compute in short time.
We run the ME test based on the code in \url{https://github.com/wittawatj/interpretable-test}.

\subsection{Implementation of Cross-Validation}\label{Appendix:cross:validation}
The candidate choices of hyper-parameters $\rho$ and $\sigma^2$ are within the set 
\[
\{(\rho,\sigma^2):~\sigma^2 = a\cdot \hat{\sigma}^2:~a\in\{0.5,1,2\}, \rho\in\{0.25,0.5,0.75\}\},
\]
where $\hat{\sigma}^2$ denotes the empirical median of pairwise distances between observations.
To choose $\rho$ and $\sigma^2$, we further split the training set into the training and validation dataset, which contain $70\%$ and $30\%$ data, respectively.
For each choice of hyper-parameters we use the training dataset to obtain a nonlinear projector and examine its hold-out performance on the validation dataset, which is quantified as the negative of the $p$-value for two-sample tests between two collection of samples in the validation dataset.
We choose hyper-parameters $\rho$ and $\sigma^2$ with the best hold-out performance.

\subsection{Tests for Synthetic Datasets}\label{Appendix:test:sync:dataset}
When studying tests on Gaussian distributions, we take both the training and testing sample sizes $N$ to be $50$.
When reproducing the experiments corresponding to the left two figures in Fig.~\ref{fig:4}, we take the dimension $D\in\{20,40,60,80,100,120,140,160\}$.
When reproducing the experiments corresponding to the right two figures, we take the sample size $n=m\in\{80,100,140,180,250\}$.

\subsection{Tests for MNIST handwritten digits}\label{Appendix:test:mnist:detail}
Table~\ref{tab:MNIST_RES1:type:I} present the type-I error for various tests in MNIST dataset, from which we can see that all tests have the type-I error close to $\alpha=0.05$.

\begin{table*}[t]
\centering
  \caption{
Average type-I error and standard error for two-sample tests in \emph{MNIST} dataset across different choices of sample size.
   } \label{tab:MNIST_RES1:type:I}
\vspace{1mm}
\begin{tabular}{c|>{\centering}p{0.15\textwidth}>{\centering}p{0.15\textwidth}>{\centering}p{0.15\textwidth}>{\centering}p{0.15\textwidth}>{\centering\arraybackslash}p{0.15\textwidth}}
\toprule
$N$ & MMD-NTK & MMD-O & ME & PW & KPW\\
\midrule
\phantom{1}$\,$200 & \mnstd{0.057}{0.0010} & \mnstd{0.056}{0.0006} & \mnstd{0.044}{0.0003} & \mnstd{0.056}{0.0004} & \mnstd{0.061}{0.0005} \\
\phantom{2}$\,$250 & \mnstd{0.051}{0.0003} & \mnstd{0.060}{0.0001} & \mnstd{0.065}{0.0002} & \mnstd{0.046}{0.0003} & \mnstd{0.048}{0.0002} \\
\phantom{1}$\,$300 & \mnstd{0.068}{0.0006} & \mnstd{0.055}{0.0003} & \mnstd{0.059}{0.0007} & \mnstd{0.056}{0.0002} & \mnstd{0.053}{0.0001} \\
\phantom{1}$\,$400 & \mnstd{0.049}{0.0007} & \mnstd{0.058}{0.0002} & \mnstd{0.041}{0.0002} & \mnstd{0.061}{0.0006} & \mnstd{0.056}{0.0006} \\
\phantom{1}$\,$500 & \mnstd{0.061}{0.0006} & \mnstd{0.054}{0.0004} & \mnstd{0.060}{0.0002} & \mnstd{0.049}{0.0003} & \mnstd{0.047}{0.0004}  \\
\midrule
Avg. & 0.057 & 0.056 & 0.053 & 0.054 & {0.053} \\
\bottomrule
\end{tabular}
\vspace{-1em}
\end{table*}

\subsection{Human activity detection}\label{Appendix:Human:data}
The pre-processing of data is as follows.
We first remove frames in which the person is standing still or with little movements.
Then we delete the first few frames to make the action of bending consist of $500$ frames.
Next we delete the last few frames to make the action of throwing consist of $355$ frames.
We take the window size $W=100$.
To perform online change point detection, we pre-train a nonlinear projector using the data before time index $300$ and compute the null statistics for many times to obtain the true threshold.
Then we compute the detection statistic by comparing the distribution between the block of data before time $300$ and the data from the sliding window.
We reject the null hypothesis and claim a change is happened if the statistic is above the threshold.
The plot of the detection statistic over time after the time index $400$ is presented in Fig.~\ref{Fig:stat:change}, and the delay detection time corresponding to all users are reported in Table~\ref{tab:MSRC:full}.

\begin{figure*}[t]
\centering
\includegraphics[width=\textwidth]{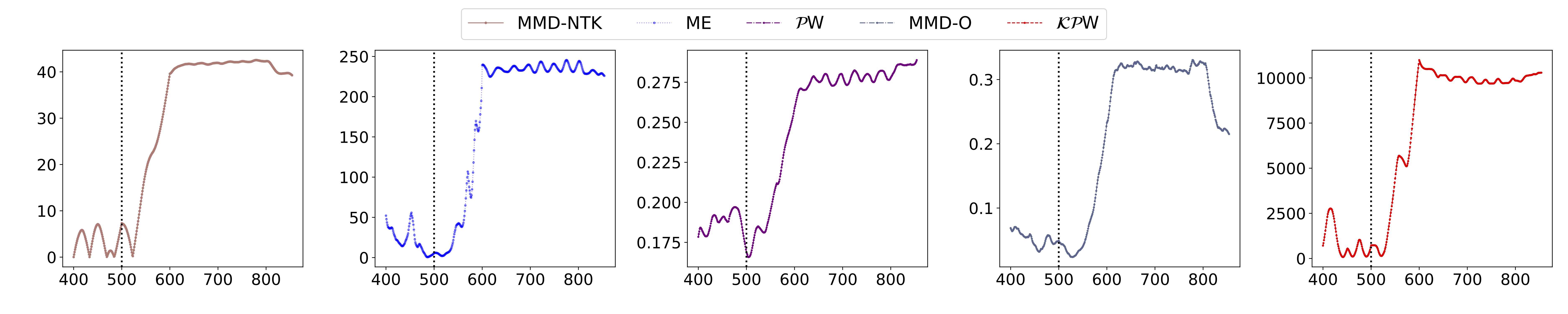}
\includegraphics[width=\textwidth]{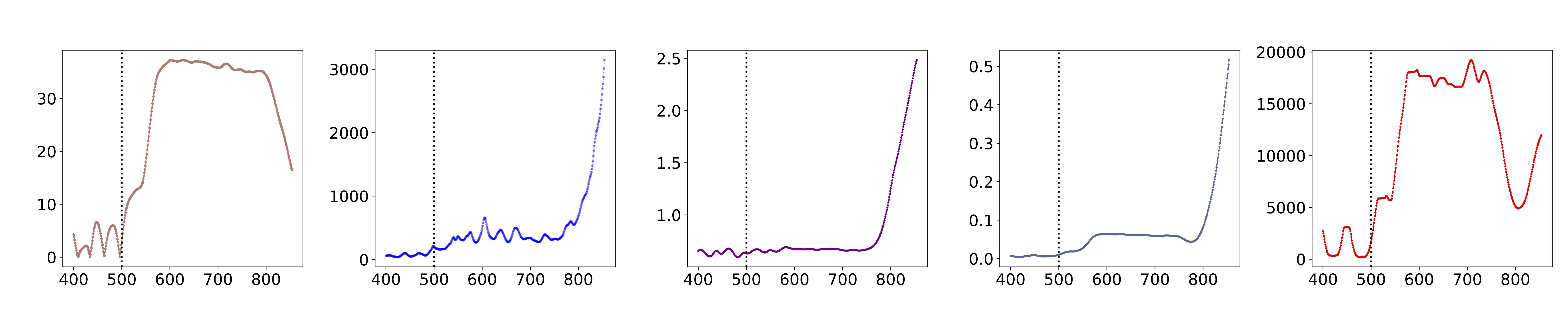}
\includegraphics[width=\textwidth]{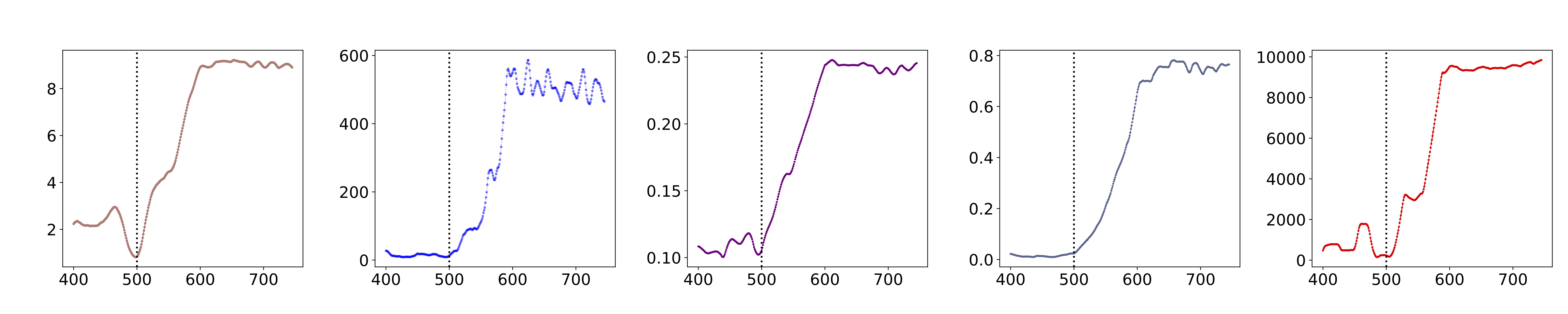}
\includegraphics[width=\textwidth]{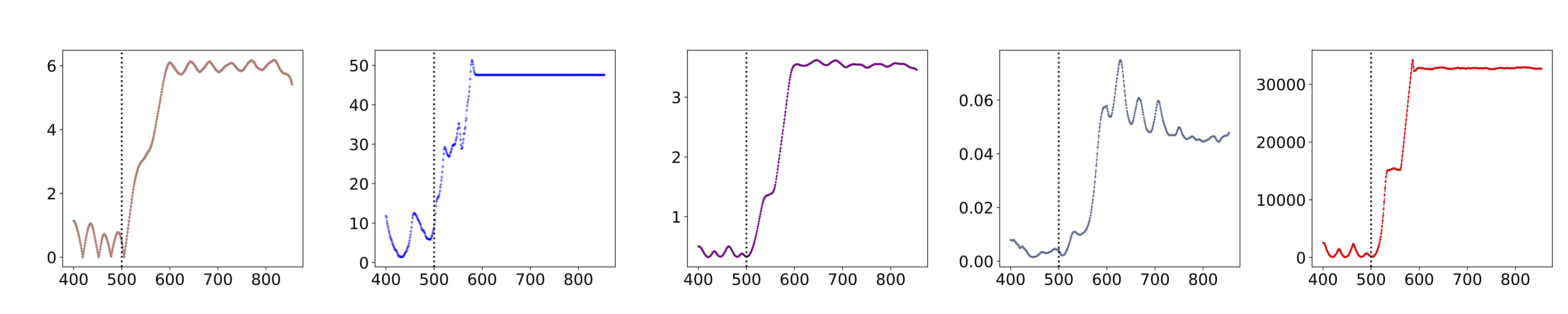}
\caption{
Comparison of detection statistics from bending to throwing for various testing procedures.
Black dash line indicates the true change-point.
Each row corresponds to detection results for each user.
}
\label{Fig:stat:change}
\end{figure*}

\section{
IMPACT OF HYPER-PARAMETERS
}\label{Appendix:impact:hyper:parameter}
\subsection{Impact of Projected Dimension \texorpdfstring{$d$}{d}}\label{Sec:H:A}
We prefer to choose the projected dimension $d$ with relatively small values since the testing statistic will have poor sample complexity rate and is expansive to compute for large $d$.
In this section, we examine the testing performance for different choices of $d$.
In particular, we perform the KPW test on Gaussian distributions (with diagonal covariance matrices, $D=128$ and $n=m=50$) and Gaussian mixture distributions (with $D=100$ and $n=m=100$) following the setup in Section~\ref{Sec:Synetic}, the results of which are reported in Fig.~\ref{Fig:choose:d}.
From the plot we can see that the testing power is generally better for $d>1$, which suggests that using vector-valued RKHS is better than using classical scalar-valued RKHS.
Moreover, we observe the performance is insensitive to the choice of $d$ as long as we take $d>1$.

\begin{figure*}[t!]
\centering
\includegraphics[width=0.7\textwidth]{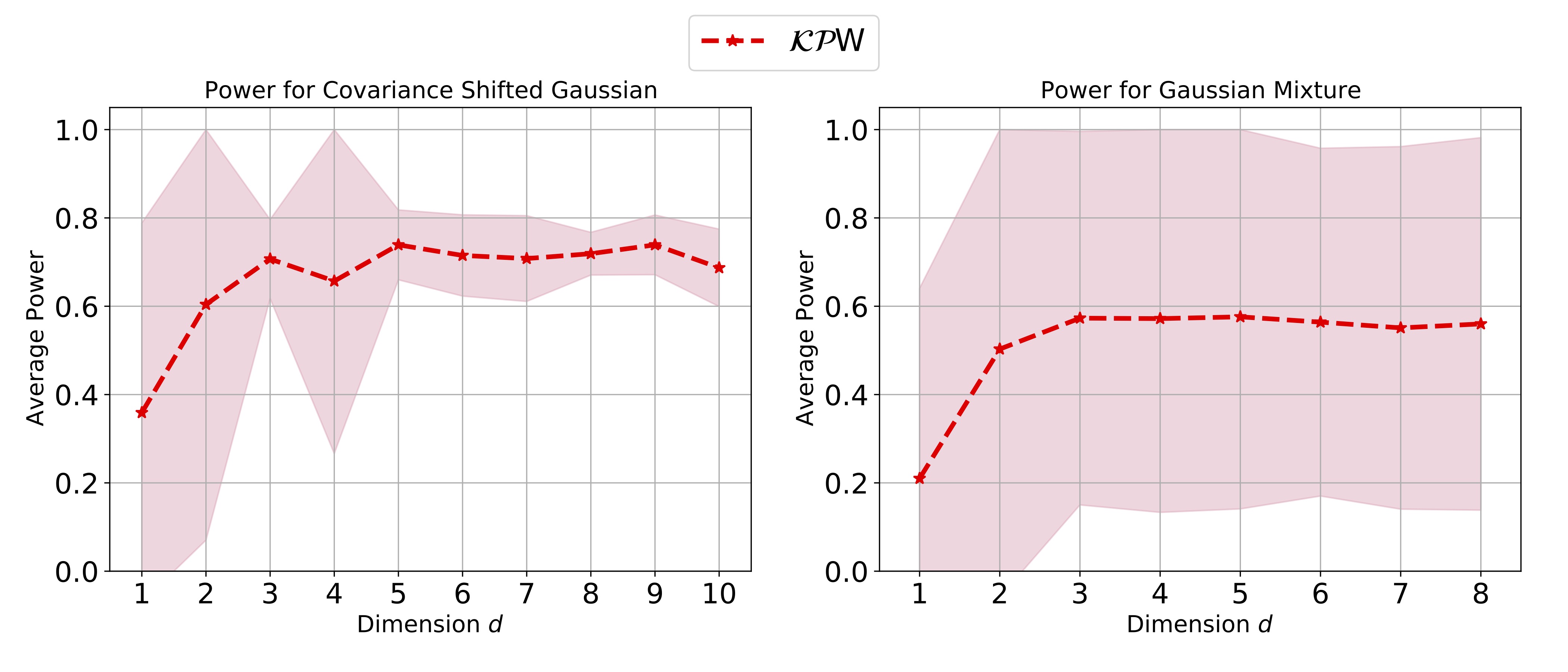}
\caption{Average power for KPW test across different choices of projected dimension $d$. Left: Gaussian distribution; Right: Gaussian mixture distribution. Results are averaged over $10$ independent trials.}
\label{Fig:choose:d}
\end{figure*}

\subsection{Impact of Entropic Regularization Parameter \texorpdfstring{$\eta$}{eta}}
As pointed out in \cite{genevay2019sample}, the entropic regularization in \eqref{Eq:max:min:entropy} could alerady improve the sample complexity result of Wasserstein distance. 
We perform experiments in this subsection to validate the impact of the entropic regularization parameter $\eta$ for the performance of KPW test.
The generated data follows Gaussian distributions (with $n=m=100$) or Gaussian mixture distributions (with $n=m=200$) with different choices of dimension $D$ and fixed sample size.
Benchmark methods include 1) \emph{KPW test with $\eta=0$} (here Wasserstein distance is computed exactly and we apply alternating optimization procedure as a heuristic); 2) \emph{Sinkhorn test} with the same $\eta$ as in the KPW test (in which we take the Sinkhorn divergence as the statistic and all training and testing samples are used); 3) \emph{Sinkhorn+} (using all data and post-selecting $\eta$ with the best performance).
Experiment results are reported in Fig.~\ref{Fig:choose:eta}, from which we can see that even Sinkhorn+ test has the curse of dimension issue.
Moreover, the KPW test with $\eta=0$ has similar performance as the KPW test.
Hence, we can assert that the KPW test is capable of alleviating the curse of dimension mainly due to the kernel projection operator instead of the entropic regularization.

\begin{figure*}[t!]
\centering
\includegraphics[width=0.7\textwidth]{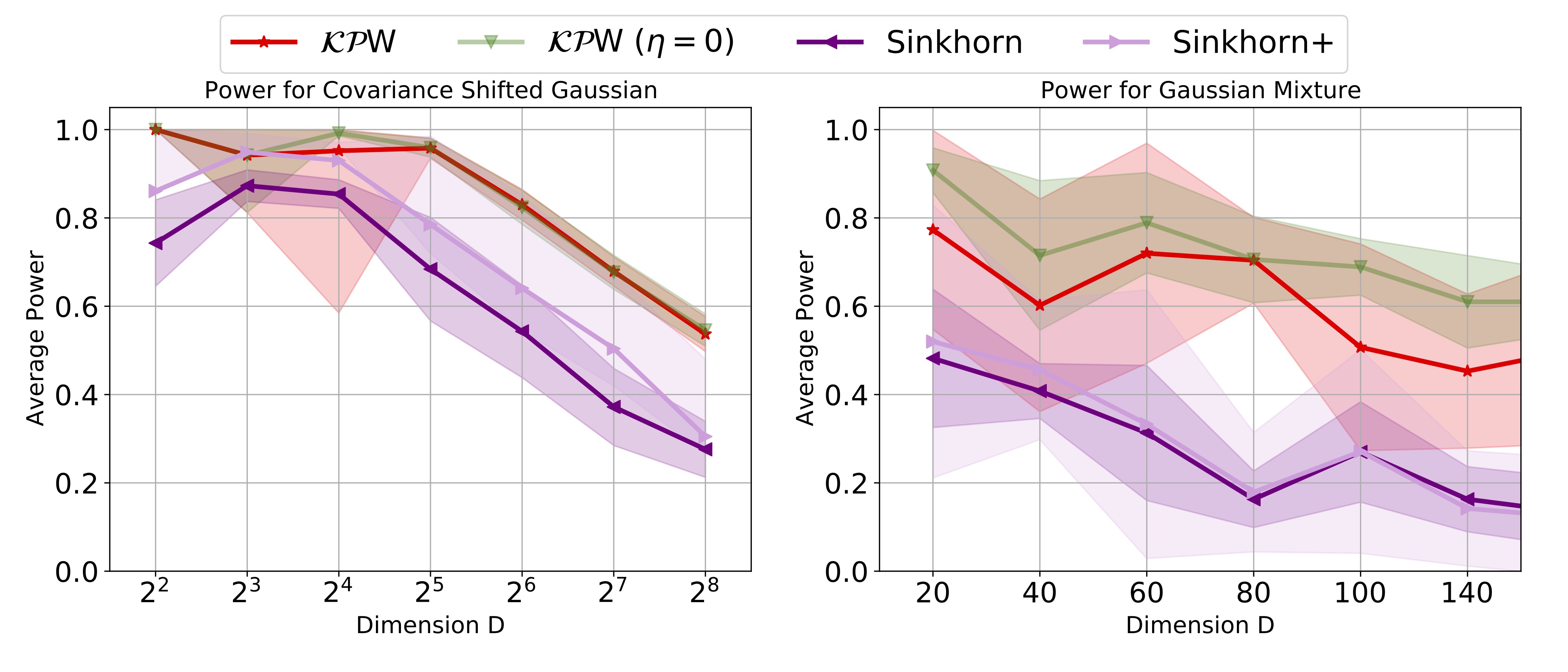}
\caption{Average power for KPW tests and Sinkhorn tests across different choices of data dimension $D$. Left: Gaussian distribution; Right: Gaussian mixture distribution. Results are averaged over $10$ independent trials.}
\label{Fig:choose:eta}
\end{figure*}

\section{
SOCIETAL IMPACT
}
Two-sample testing is not only a fundamental problem in statistics but also growing increasing attention in machine learning.
On the one hand, it plays a key role in modern applications such as anomaly detection and health care.
On the other hand, it can help to design better algorithms for artificial intelligence such as GANs.
Our work shows a competitive performance for dealing with high-dimensional data by nonlinear dimensionality reduction using kernel trick.
It identifies the difference between two collections of samples by extracting the most representative nonlinear features. 
We hope this work can be applied to design more powerful algorithms in those areas.

\end{document}